\documentclass[11pt]{article}

\usepackage[ansinew]{inputenc}
\usepackage{float}
\usepackage{graphicx}
\usepackage{fancyhdr}
\usepackage{amsmath,amssymb}
\usepackage{amsthm}
\usepackage{color}
\usepackage{verbatim}
\usepackage{scrtime}
\usepackage{float}
\usepackage{subfigure}
\usepackage{url}
\usepackage{multirow}
\usepackage[usenames,dvipsnames]{xcolor}
\usepackage{tabularx}
\usepackage{framed}
\usepackage{lipsum} 
\usepackage{a4wide}
\usepackage{epstopdf}
\usepackage[raggedright]{titlesec} 

\usepackage{hyperref}

\hypersetup{
     colorlinks   = true,
     citecolor    = blue
}

\usepackage[sort&compress, numbers]{natbib}
\usepackage{enumerate}

\usepackage{algpseudocode}
\usepackage{algorithm}
\algtext*{EndIf}
\algnewcommand\algorithmicinput{\textbf{Input:}}
\algnewcommand\Input{\item[\algorithmicinput]}
\algnewcommand\algorithmicoutput{\textbf{Ouput:}}
\algnewcommand\Output{\item[\algorithmicoutput]}

\usepackage{afterpage}

\newtheorem{theorem}{Theorem}[section]
\newtheorem{proposition}[theorem]{Proposition}

\newcommand{\trace}{{\rm Trace}}

\newcommand{\D}{\mathrm{D}}

\newcommand{\Stiefel}[2]{{\mathrm{St}({#1},{#2})}}
\newcommand{\Gstiefel}[3]{{\mathrm{St}_{#3}({#1},{#2})}}
\newcommand{\OG}[1]{{\mathcal{O}({#1})}}
\newcommand{\GL}[1]{{\mathrm{GL}({#1})}}
\newcommand{\Sym}{{\mathrm{Sym}}}

\newcommand{\hess}{\mathrm{Hess}}
\newcommand{\mat}[1]{{\bf #1}}
\newcommand{\subject}{\mathrm{subject\  to}}
\newcommand{\PSD}[2]{\mathrm{S}_+({#1},{#2})}
\newcommand{\Grass}[2]{\mathrm{Gr}({#1},{#2})}

\newcommand{\grad}{\mathrm{grad}}

\newcommand{\argmin}{\operatornamewithlimits{arg\,min}}

\newcommand{\hlift}{{\rm horizontal\ lift\ of\ }}

\newcommand{\Fro}{{F}} 

\newcommand{\change}[1]{#1}
\newcommand{\changeBM}[1]{#1}
\newcommand{\changeBMM}[1]{#1}

\title{Riemannian preconditioning\thanks{This paper presents research results of the Belgian Network DYSCO (Dynamical Systems, Control, and Optimization), funded by the Interuniversity Attraction Poles Programme, initiated by the Belgian State, Science Policy Office. The scientific responsibility rests with its authors. This work was also partly supported by the Belgian FRFC (Fonds de la Recherche Fondamentale Collective).}}

\author{Bamdev Mishra\thanks{Amazon Development Centre India, Bangalore 560055, India (Bamdevm@amazon.com). This
author's research was supported as a research fellow of the Belgian National Fund for Scientific Research (FNRS). Most of the work was done while this author was with the Department of Electrical Engineering and Computer Science, University of Li\`ege, 4000 Li\`ege, Belgium, and was visiting the Department of Engineering (Control Group), University of Cambridge, Cambridge CB2 1PZ, UK.}
        \and Rodolphe Sepulchre\thanks{Department of Engineering, University of Cambridge, Cambridge CB2 1PZ, UK, and Depart- ment of Electrical Engineering and Computer Science, University of Li\`ege, 4000 Li\`ege, Belgium (R.Sepulchre@eng.cam.ac.uk).} 
        }

\date{Compiled on \today}

\begin{document}
\maketitle

\begin{abstract}
Thispaper exploits a basic connection between sequential quadratic programming and Riemannian gradient optimization to address the general question of selecting a \emph{metric} in Riemannian optimization, \change{in particular when the Riemannian structure is sought on a \emph{quotient} manifold}. The proposed method is shown to be particularly insightful and efficient in quadratic optimization with orthogonality and/or rank constraints, which covers most current applications of Riemannian optimization in matrix manifolds.
\end{abstract}

%



\section{Introduction}

Gradient algorithms are a method of choice for large-scale optimization, but their convergence properties critically depend on the choice of a suitable {\it metric}. Good metrics can lead to superlinear convergence whereas bad metrics can lead to very slow convergence. Goodness of the metric depends on its ability to encode second-order information about the optimization problem. For general optimization problems with equality constraints, sequential quadratic programming (SQP) methods provide an efficient selection procedure based on (approximating) the Hessian of a local quadratic approximation of the problem \citep[Chapter~18]{nocedal06a}. This approach is Lagrangian; that is, it lifts the constraint into the cost function. An alternative is to embed the constraint into the search space, leading to unconstrained optimization on a nonlinear search space. Selecting the metric then amounts to equipping the search space with a Riemannian structure \citep{smith94a, edelman98a, absil08a}. Riemannian optimization has gained popularity in the recent years because of the particular nature of the constraints that show up in many matrix or tensor applications, in particular \emph{orthogonality} and \emph{rank} constraints. Such constraints are \changeBMM{nonlinear} and \changeBMM{nonconvex}, but nevertheless very special \citep{edelman98a, mishra14a}. In particular, they are efficiently encoded through matrix factorizations and their underlying geometry makes the search space sufficiently structured to make the machinery of Riemannian optimization competitive with alternative approaches, including convex relaxations. A current limitation of Riemannian optimization is, however, in the choice of the metric. Previous work has mostly focused on the search space, exploiting the differential geometry of the constraint, but disregarding the role of the cost function. This limitation was pointed out early \citep{manton02a} and has been addressed in a number of recent contributions that emphasized the importance of preconditioning \citep{mishra14c, mishra12a, mishra14d, vandereycken10a, ngo12a}, but with no general procedure. The simple observation, and the main contribution, of the present paper is that SQP provides a systematic framework for choosing a metric in Riemannian optimization in a way that takes into consideration both the cost function and the constrained search space. This connection seems novel and insightful, \change{especially in the situation where the unconstrained optimization problem is formulated on a quotient manifold, leading to a situation where the Hessian of the Lagrangian is singular in the total space. Most notably, this situation covers optimization problems on the Grassmann manifold and on the manifold of  matrices of fixed rank.}

This paper advocates that the use of SQP to select a metric in Riemannian optimization is  general and connects two rather independent areas of constrained optimization. We focus in particular on the special case of quadratic cost functions with orthogonality and/or rank constraints. This particular situation encompasses a great deal of current successful applications of Riemannian optimization, including the popular generalized eigenvalue problem \citep{edelman98a, absil02a} and linear matrix equation problems \citep{benner13a, vandereycken10a}. Even in these highly researched problems, we show that SQP methods unify a number of recent disparate results and provide novel metrics. In the eigenvalue problem, where both the cost and constraints are quadratic, the SQP method suggests a parameterized family of Riemannian metrics that provides novel insight on the role of shifts in the power, inverse, and Rayleigh quotient iteration methods. In the problem of solving linear matrix equations, low-rank matrix factorizations make the cost function quadratic in each of the factors, leading to Riemannian metrics rooted in block-diagonal approximations of the Hessian. In all of the mentioned applications, we stress the complementary, but not equivalent, role of SQP and Riemannian optimization: the SQP method provides a \change{framework for selecting the metric from the (degenerate) Lagrangian in the total space} while the Riemannian framework provides the necessary generalization of unconstrained optimization to quotient manifolds, allowing for rigorous design and convergence analysis of a broad class of quasi-Newton algorithms in optimization problems over {\it classes of equivalences} of matrices.

We view this approach of selecting a metric from the Lagrangian as a form of \emph{Riemannian preconditioning}. Similar to the notion of preconditioning in the unconstrained case \citep[Chapter~5]{nocedal06a}, the chosen Riemannian metrics have a preconditioning effect on optimization algorithms.

This paper does not aim at a comprehensive treatment of the topic, but rather focuses on connections between several classical branches of matrix calculus: matrix factorizations and shifts in numerical linear algebra, Riemannian submersions in differential geometry, and SQP in constrained optimization. The connection between SQP and the Riemannian Newton method on submanifolds is shown in Section \ref{sec:connection_SQP_Newton}. The general idea of using the Lagrangian to select a metric in Riemannian optimization for constraints with symmetries is presented in Section \ref{sec:ingredients}. We discuss the choice of the metric depending on the curvature properties of both the cost and the constraint and the interpretation of the Lagrange parameter as a shift. Section \ref{sec:quadratic_orthogonality} develops the specific situation of quadratic cost and orthogonality constraints, revisiting the classical generalized eigenvalue problem. Section \ref{sec:quadratic_lowrank} further develops the specific situation of quadratic cost and rank constraints, with applications to solving large-scale matrix Lyapunov equations. All numerical illustrations use the Matlab toolbox Manopt \citep{boumal14a}.


\section{SQP as Riemannian Newton method}\label{sec:connection_SQP_Newton}

\begin{figure}[t]
\centering
\includegraphics[scale = 0.55]{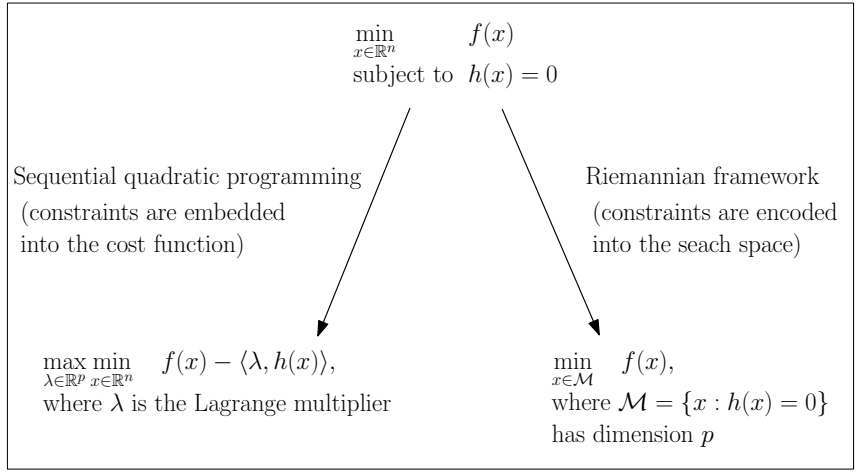}
\caption{Two complementary viewpoints on optimization with equality constraints.}
\label{fig:sqp_manifold}
\end{figure}

Consider the optimization problem
\begin{equation}\label{eq:optimization_problem}
\begin{array}{lll}
\min\limits_{x \in \mathbb{R}^n} & f (x) \\
\subject & h(x) = 0,
\end{array}
\end{equation}
where $f:\mathbb{R}^n\rightarrow \mathbb{R}$ and $h: \mathbb{R}^n \rightarrow \mathbb{R}^p$ are smooth functions. \change{To illustrate the connection between SQP and Riemannian optimization, we first consider the simplified situation in which} the set $\mathcal{M} := \{x : h(x) =0 \}$ has the structure of an \emph{embedded differentiable submanifold} of $\mathbb{R}^n$ \citep[Section~3.3]{absil08a}, i.e., \change{$h_x(x)$ is full rank everywhere, where $h_x(x)$ is the first-order derivative of $h$ at $x$ and is represented as an $n\times p$ matrix}.

\change{We recall} two complementary approaches for the problem (\ref{eq:optimization_problem}), namely the sequential quadratic approach and the Riemannian approach. The schematic view is shown in Figure \ref{fig:sqp_manifold}. We aim at connecting these two approaches in order to tune the metric on the search space in such a way that it incorporates \emph{second-order} information of the problem.

\subsection{The constrained optimization viewpoint (SQP)}\label{sec:sqp}
SQP is a particularly well known approach for equality constrained nonlinear optimization \citep[Chapter~18]{nocedal06a}. The core idea is to deal with the unconstrained \emph{Lagrangian} $ {\mathcal{L}}: \mathbb{R}^{n} \times \mathbb{R}^p \rightarrow \mathbb{R}: ( {x},  {\lambda}) \mapsto  {\mathcal{L}}( {x},  {\lambda})$ with two sets of parameters, defined as
\begin{equation}\label{eq:Lagrangian}
 {\mathcal{L}}( {x},  {\lambda} ) =  {f}(  x) - \langle  {\lambda},  {h}(  x) \rangle,
\end{equation} 
where $\langle \cdot, \cdot \rangle$ is the Euclidean inner product and $ {\lambda} \in \mathbb{R}^p$ is the \emph{Lagrange multiplier}. With the introduction of $\lambda$, the optimization problem (\ref{eq:optimization_problem}) translates to the problem
\begin{equation}\label{eq:maxmin_Lagrangian}
\max\limits_{\lambda \in \mathbb{R}^p} \min\limits_{x\in \mathbb{R}^n} \quad \mathcal{L}(x, \lambda),
\end{equation}
which leads to a \emph{primal-dual} iterative algorithm in $(x, \lambda) \in \mathbb{R}^n \times \mathbb{R}^p$, but the linearity in $\lambda$ is exploited to reduce the number of variables. For example, in the neighborhood of a local minimum, the best \change{\emph{least-squares estimate}} of the Lagrange multiplier is 
\begin{equation}\label{eq:least_square_lambda}
\change{{\lambda}_x= (  {h}_x(  x)^T\allowbreak  ( {h}_x(  x)) )^{-1} \allowbreak {h}_x(  x) ^T {f} _x(  x),}
\end{equation}
where $ {h}_x(  x)$ and $f _x(x)$ are the first-order derivatives of the functions $ {h}$ and $f$ at $ {x}$, respectively \citep[Equation~(18.21)]{nocedal06a}. \change{We use the convention that $f _x(x)$ is a column vector of length $n$ and $h_x(x)$ is a full column rank matrix of size $n\times p$. The least-squares estimate of $\lambda$ corresponds to the least-squares solution of the \emph{optimality condition} $h_x(x) \lambda = f_x(x)$. It should be noted that any other estimate for $h_x(x) \lambda = f_x(x)$ is equally valid, provided that the equality is satisfied at optimality.} 

Substituting $\lambda$ with $\lambda_x$ transforms the primal-dual iteration in $( {x}, {\lambda})$ to a purely primal iteration in the variable ${x}$ alone \citep[Page~539]{nocedal06a}. Once the Lagrange multiplier is estimated, the SQP optimization approach proceeds by minimizing the quadratic programming problem
\begin{equation}\label{eq:sqp}
\begin{array}{lll}
\arg\min\limits_{{  \zeta}_{  x} \in \mathbb{R}^n} &  {f}(  x) + \langle  {f}_x (  x),  {\zeta}_{  x} \rangle + \frac{1}{2}\langle  {\zeta}_{  x}, \D^2 {\mathcal{L}}( {x},  {\lambda}_{  x}) [ {\zeta}_{  x}]\rangle \\
\subject & {\rm D}  {h}(  x)[ {\zeta}_{  x}] = 0
\end{array}
\end{equation}
at each iteration, \change{where $x$ is assumed to be in the constraint set (i.e., it is a feasible point)}, $ {f}_x(  x)$ is the first-order derivative of the cost function $f$, \change{$ \D^2 {\mathcal{L}}( {x},  {\lambda}_{  x}) [ {\zeta}_{  x}]$} is the second-order derivative of $\mathcal{L}(x, \lambda_x)$ with respect to $x$ (keeping $ {\lambda}_{  x}$ fixed) that is applied in the direction $\zeta_x \in \mathbb{R}^n$, $\langle \cdot, \cdot \rangle$ is the Euclidean inner product, and ${\rm D} {h}(  x)[{  \zeta}_{  x}]$ is the Euclidean directional derivative of $ {h}(  x)$ in the direction $ {\zeta}_{  x} \in \mathbb{R}^n$, i.e., ${\rm D}  {h}(x)[ {\zeta}_{  x}] = \lim_{t \rightarrow 0} ( {h} ({ {x} + t  {\zeta}_{  x} }) -  h ({  x}))/t$.

If the quantity $ \langle  {\zeta}_{  x}, \D^2 {\mathcal{L}}( {x},  {\lambda}_{  x}) [ {\zeta}_{  x}]\rangle$ is strictly positive in the tangent space of constraints, then the problem (\ref{eq:sqp}) has a \emph{unique solution} \citep[Section~18.1]{nocedal06a}. The solution $ {\zeta}^*_{  x}$ of (\ref{eq:sqp}) is interpreted as a \emph{search} direction. Following the search direction \changeBM{with a step-size $s$}, the next iterate $ {x}_+$ in the SQP algorithm is obtained by projecting $ {x} +  s{\zeta}^*_{  x}$ onto the constrained space to maintain strict feasibility of the iterates, where $ {\zeta}^*_{  x}$ is the solution to (\ref{eq:sqp}). \change{This specific SQP algorithm is also called the \emph{feasibly projected} SQP method as every iterate $x_+$ is feasible \citep{absil09b,wright04a}}.

The resulting iterative algorithm is shown in Table \ref{tab:sqp} and has the properties of a Newton algorithm (locally in the neighborhood of a minimum) with favorable convergence properties \citep{wright04a}. \change{The algorithm in Table \ref{tab:sqp} is the simplified version of the algorithm in \citep{wright04a}.} Additionally, the SQP approach is appealing for the simplicity of its formulation.



\begin{table}[t]
\begin{center} \small
\begin{tabular}{ |p{12cm}| }
\hline
\\
The SQP algorithm for 
\[
\begin{array}[t]{lll}
\min\limits_{x \in \mathbb{R}^n} & f(x)\\
\subject & h(x) = 0.
\end{array}
\]
\begin{enumerate}
\item Compute the search direction ${  \zeta}_{  x}^*$ that is the solution of (\ref{eq:sqp}).
\item The next iterate $ {x}_+$ is obtained by projecting $ {x} +  s{\zeta}^*_{  x}$ onto the constrained space, where the step-size $s$ is obtained with backtracking line search.
\item Repeat until convergence.
\end{enumerate}
  \\
 \hline
\end{tabular}
\end{center} 
\caption{The SQP algorithm.}
\label{tab:sqp} 
\end{table}

\subsection{Connecting SQP to Riemannian Newton method}
An alternative treatment of problem (\ref{eq:optimization_problem}) is to recast it as an unconstrained optimization problem on a nonlinear search space that encodes the constraint. For special constraints that are sufficiently structured, the framework leads to an efficient computational framework \citep{absil08a, edelman98a, smith94a}. Particularly, concrete numerical algorithms, both first-order and second-order, are developed by endowing the set $\mathcal{M} :=\{x : h(x) =0 \}$ with a \emph{Riemannian structure}.

The connection between SQP and the Riemannian framework for constraints that are embedded in $\mathbb{R}^n$ (i.e., the manifold $\mathcal{M}$ has the structure of an embedded submanifold) has been studied in \citep{gabay82a}, \citep[Section~4.9]{edelman98a}, and more recently in \citep[Section~4]{absil09b}. In particular, the following two results are relevant for the present paper.

\change{
\begin{proposition}\label{prop:connecting_SQP_Newton}
(\citep[Proposition~4.1]{absil09b}): \changeBM{If} $\mathcal{M}$ is an embedded submanifold of $\mathbb{R}^n$, then the SQP algorithm in Table \ref{tab:sqp} is \changeBM{equivalent to} the Riemannian Newton method for the choice of the Lagrange multiplier $\lambda$ in (\ref{eq:least_square_lambda}).
\end{proposition}
}

\change{
\begin{proposition}\label{prop:second_order_SQP_Newton}
(From the remark made by \citet[Remark~4.3]{gabay82a}): \changeBM{Assume} that $\mathcal{M}$ is an embedded submanifold of $\mathbb{R}^n$ and $f:\mathcal{M} \rightarrow \mathbb{R}$ is a smooth function with isolated minima on $\mathcal{M}$. If $x^* \in \mathcal{M}$ is a local minimum of $f: \mathcal{M} \rightarrow \mathbb{R}$ on $\mathcal{M}$, then the second-order derivative of the Lagrangian along $\zeta_{x^*} \in T_{x^*} \mathcal{M}$ (the tangent space of $\mathcal{M}$ at $x^*$), i.e., $\langle \zeta_{x^*}, \D^2\mathcal{L}_x(x^*, \lambda_{x^*})[\zeta_{x^*}] \rangle$, captures \changeBM{all} second-order information of the cost function $f$ on $\mathcal{M}$.
\end{proposition}
}

%
%
%

Propositions \ref{prop:connecting_SQP_Newton} and \ref{prop:second_order_SQP_Newton} show that SQP and the Riemannian Newton method are equivalent for embedded submanifolds. Furthermore, this connection emphasizes the role of the Hessian of the Lagrangian in the total space in effectively capturing second-order information of the problem, locally in the neighborhood of a minimum. 

\change{Our goal in the present paper is to build upon this connection in situations where the constrained search space is equipped with the differentiable structure of a {\it quotient manifold} rather than an {\it embedded submanifold}. This situation is frequent in applications due to symmetry properties of the optimization problem. The symmetries make the connection between SQP and Riemannian optimization less straightforward  because the Lagrangian is \emph{degenerate} at a local minimum. Nevertheless, we aim at showing that the structure of the Lagrangian in the total space is still very insightful in suggesting Riemannian metrics that capture second-order information on  the quotient manifold.}


\section{Metric selection on quotient manifolds}\label{sec:ingredients}

\change{Quotient manifolds are embedded in the \changeBM{manifold $\mathcal{M}$ (that itself is in $\mathbb{R}^n$)} up to an \emph{equivalence relation} $\sim$. The search space is then the set of equivalence classes. Optima are not isolated in the computational \emph{total} space $\mathcal{M}$, but they become isolated once we consider the abstract \emph{quotient} space $\mathcal{M}/\sim$. (\changeBM{A precise} characterization of the space $\mathcal{M}/\sim$ is shown in Section \ref{sec:Riemannian_framework}.)}

\change{Because of their prevalence in applications, quotient manifolds have been a focus of much research in Riemannian optimization \citep{smith94a, edelman98a, absil08a}. Algorithms are run in the total space, but under appropriate compatibility between the Riemannian structure of the total space and the Riemannian structure of the quotient manifold, they define algorithms in the quotient space.} \changeBM{The exposition for quotient manifolds here follows \citep[Chapters~3,~5, and~8]{absil08a}}.


\subsection{The Riemannian optimization viewpoint}\label{sec:Riemannian_framework}


Consider an equivalence relation $\sim$ in the \emph{total} (computational) space $ {\mathcal{M}}$. The quotient manifold ${\mathcal M}/\sim$ generated by this equivalence property consists of elements that are \emph{equivalence classes} of the form $ [ {x}] =  \{ {y} \in  {\mathcal M} :  {y} \sim  {x }\}$. Equivalently, if $[x]$ is an element in $\mathcal{M}/\sim$, then its matrix representation in $\mathcal{M}$ is $x$. For example, the Grassmann manifold $\Grass{r}{n}$, which is the set of $r$-dimensional subspaces in $\mathbb{R}^n$, is obtained by the equivalence relationship $\Grass{r}{n} = \Stiefel{r}{n}/\OG{r}$. $\Stiefel{r}{n}$ is the set of matrices of size $n\times r$ with \changeBMM{orthonormal} columns and $\OG{r}$ is the set of square $r\times r$ \emph{orthogonal matrices}. Each element in the total space $\mathcal{M} := \Stiefel{r}{n} = \{\mat{X} \in \mathbb{R}^{n\times r}: \mat{X}^T\mat{X} = \mat{I} \}$ is characterized by a matrix $\mat{X} \in \mathbb{R}^{n \times r}$ such that $\mat{X}^T \mat{X} = \mat{I}$. And an abstract element on the Grassmann manifold $\Grass{r}{n} $  is characterized by the equivalence class $[\mat{X}] := \{\mat{X\mat{O}}: \mat{O} \in \OG{r} \}$ at $\mat{X} \in \Stiefel{r}{n}$.

\begin{figure}[t]
\subfigure{
\includegraphics[scale = 0.33]{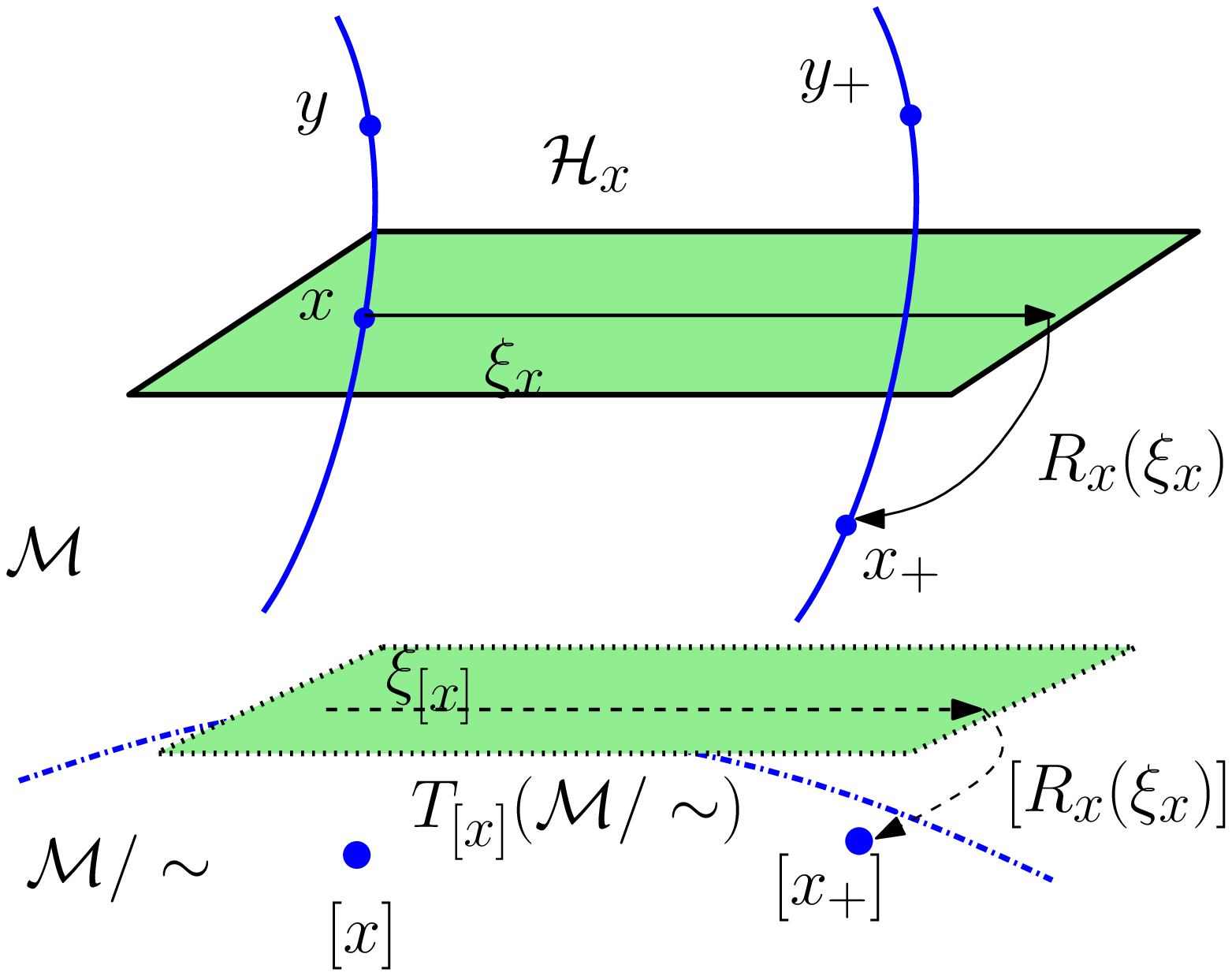}
}
\hspace{-0.5em}
\subfigure{
\includegraphics[scale = 0.56]{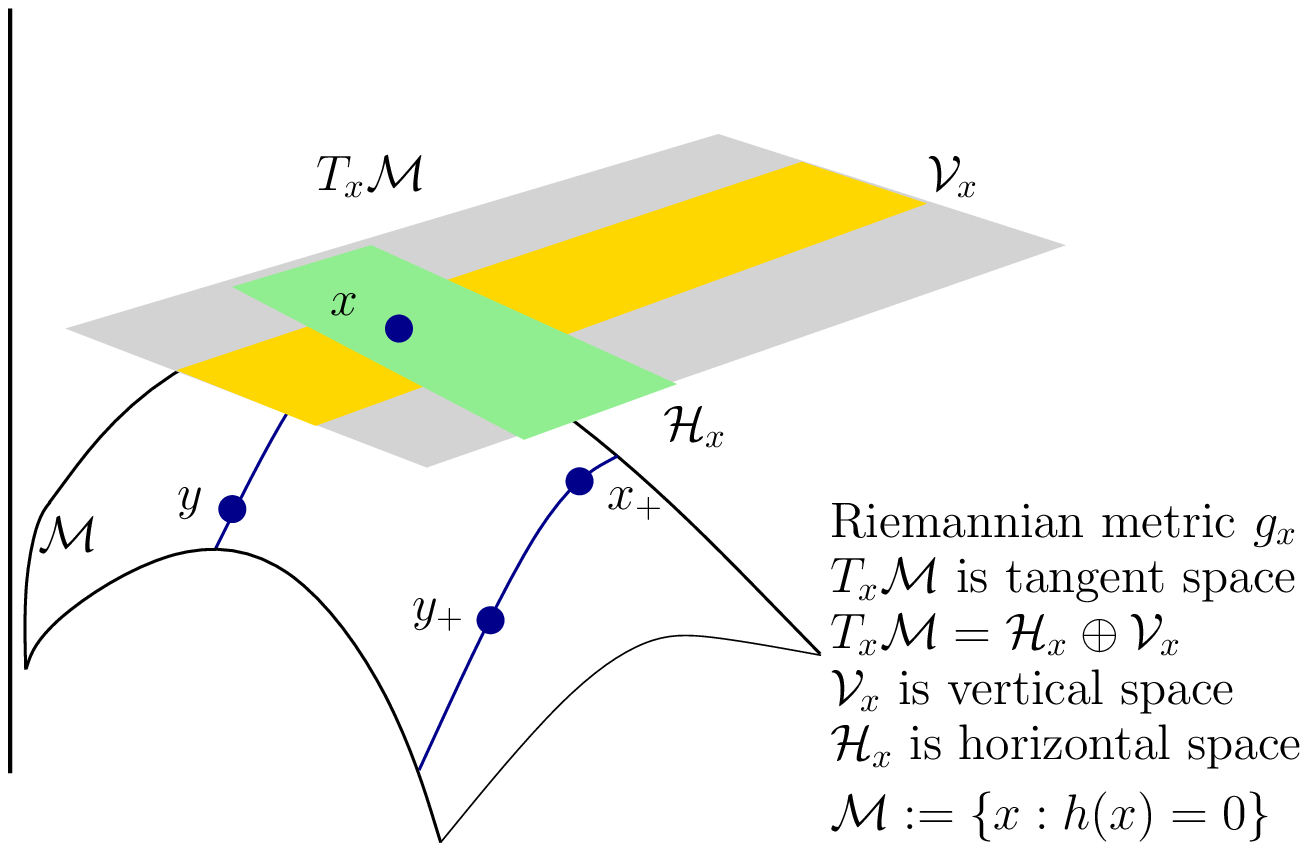}
}
\caption{A schematic view of optimization on constraints with symmetries. \change{The dotted lines represent abstract objects and the solid lines are their matrix representations.} The points $ {x}$ and $ y$ in the total (computational) space $ { \mathcal{M}}$ belong to the same equivalence class (shown in solid blue) and they represent a single point $ [ {x}] :=  \{ {y} \in  {\mathcal M} :  {y} \sim  {x }\}$ in the quotient space $\mathcal{M}/\sim$. An algorithm by necessity is implemented in the computation space, but conceptually, the search is on the quotient manifold. With the Riemannian metric $g$ from (\ref{eq:metric_quotient}), the quotient manifold $\mathcal{M}/\sim$ is \emph{submersed} into $\mathcal{M}$. The vertical space $\mathcal{V}_x$ is the linearization of the equivalence class. The horizontal space $\mathcal{H}_{ {x}}$ is complementary to the vertical space $\mathcal{V}_x$ and provides a matrix representation to the abstract tangent space $T_{{[x]}}( {\mathcal{M}}/\sim)$ of the Riemannian quotient manifold. Consequently, tangent vectors on the quotient space are lifted to the horizontal space. \change{Given $ {\xi}_{  x}$ as the horizontal lift, i.e., matrix representation of a tangent vector $\xi_{[x]}$ belonging to the abstract space $T_{[x]} (\mathcal{M}/\sim)$, $ {R}_{  x}$ maps it onto an element in $\mathcal{M}$.}}
\label{fig:manifold_optimization}
\end{figure}

Since the manifold $\mathcal{M}/\sim$ is an abstract space, the elements of its tangent space $T_{[x]} (\mathcal{M}/\sim)$ at $[x]$ also call for a matrix representation in \changeBM{the tangent space $T_x {\mathcal{M}}$} that respects the equivalence \change{relation} $\sim$. Equivalently, \change{the} matrix representation of $T_{[x]} (\mathcal{M}/\sim)$ should be restricted to the directions in the tangent space $T_{  x }  {\mathcal{M}}$ on the total space $ {\mathcal M}$ at ${  x}$ that do not induce a displacement along the equivalence class $[x]$. This is realized by decomposing $T_{  x}  {\mathcal M}$ into complementary subspaces, the \emph{vertical} and \emph{horizontal} subspaces such that $ \mathcal{V}_{  x}  \oplus \mathcal{H}_{  x} = T_{  x}  {\mathcal M}$. The vertical space $\mathcal{V}_{  x}$ is the tangent space of the equivalence class $[x]$. On the other hand, the horizontal space $\mathcal{H}_{  x}$, which is \change{any complementary subspace} to $\mathcal{V}_{  x}$ in $T_x\mathcal{M}$, provides a valid matrix representation of the abstract tangent space $T_{[x]} (\mathcal{M}/\sim)$ \citep[Section~3.5.8]{absil08a}. An abstract tangent vector $\xi_{[x]} \in T_{[x]} (\mathcal{M}/\sim)$ at $ [{x}]$ has a unique element in the horizontal space $ {\xi}_{ {x}}\in\mathcal{H}_{ {x}}$ that is called its \emph{horizontal lift}. \change{Our specific choice of the horizontal space is the subspace of $T_x \mathcal{M}$ that is the \emph{orthogonal complement} of $\mathcal{V}_{  x}$ in the sense of a Riemannian metric}.

\change{A Riemannian metric $ {g}_{  x} : T_x \mathcal{M} \times T_x \mathcal{M} \rightarrow \mathbb{R}$ at $ {x} \in \mathcal{  M}$ in the total space defines a Riemannian metric $g_{[x]}:T_{[x]} (\mathcal{M}/\sim) \times T_{[x]} (\mathcal{M}/\sim) \rightarrow \mathbb{R}$, i.e.,
\begin{equation}\label{eq:metric_quotient}
	g_{[x]}(\xi_{[x]},\eta_{[x]}):= {g}_{ {x}}( {\xi}_{ {x}}, {\eta}_{ {x}}),
\end{equation}
on the quotient manifold $\mathcal M/\sim$, provided that the expression ${g}_{ {x}}( {\xi}_{ {x}}, {\eta}_{ {x}})$ does not depend on a specific representation along the equivalence class $[x]$. Here $\xi_{[x]}$ and $\eta_{[x]}$ are tangent vectors in $T_{[x]} (\mathcal{M}/\sim)$, and $ {\xi}_{ {x}},  {\eta}_{ {x}}$ are their horizontal lifts in $\mathcal{H}_{ {x}}$ at $x$.} Equivalently, the definition (\ref{eq:metric_quotient}) is well posed when $ {g}_{  x}( {\xi}_{ {x}},  {\zeta}_{ {x}}) =  {g}_{  y} (  {\xi}_{ {y}},  {\zeta}_{ {y}})$ for all $ {x},  {y} \in [  x]$, where $ {\xi}_{ {x}},  {\zeta}_{ {x}} \in \mathcal{H}_{ {x}}$ and $ {\xi}_{ {y}},  {\zeta}_{ {y}} \in \mathcal{H}_{ {y}}$ are the horizontal lifts of $\xi_{[x]}, \zeta_{[x]} \in T_{[x]} (\mathcal{M}/\sim)$ along the same equivalence class $[x]$ \citep[Section~3.6.2]{absil08a}. In words, the metric $g_x$ is \emph{invariant} along the equivalence class $[x]$. Endowed with this Riemannian metric, the quotient manifold $\mathcal{M}/\sim$ is called a \emph{Riemannian quotient manifold} of $ {\mathcal M}$.

The choice of the metric (\ref{eq:metric_quotient}), which is invariant along the equivalence class $[ {x}]$, and of the horizontal space $\mathcal{H}_x$ as the orthogonal complement of $\mathcal{V}_{  x}$, in the sense of the Riemannian metric (\ref{eq:metric_quotient}), makes the space $\mathcal{M}/\sim$ a \emph{Riemannian submersion}. It allows for a convenient matrix representation of the gradient of a cost function. Figure \ref{fig:manifold_optimization} presents a schematic view of the search space. Consequently, the steepest-descent algorithm on the manifold $\mathcal{M}$ that respects the equivalence property $\sim$ on the space acquires the form shown in Table \ref{tab:Riemannian_steepest_descent}. Convergence of the steepest-descent algorithm in the neighborhood of a local minimum is shown by \citet[Theorems~4.3.1 and 4.5.1]{absil08a}. The main ingredients of Table \ref{tab:Riemannian_steepest_descent} are the gradient direction computation and the \emph{retraction} mapping.

\begin{table}[t]
\begin{center} \small
\begin{tabular}{ |p{12cm}| }
\hline
\\
The Riemannian steepest-descent algorithm for 
\[
\begin{array}[t]{lll}
\min\limits_{x \in \mathcal{M}} & f(x).\\
\end{array}
\]
\begin{enumerate}
\item Search direction: compute the negative Riemannian gradient $ {\xi}_x = -  {\grad}_{ { x}} {  f}$ with respect to the Riemannian metric $ {g}_{  x}$ (\ref{eq:metric_quotient}). \change{Equivalently, we solve the problem (\ref{eq:Riemannian_gradient_computation_qp}).}
\item Retract with backtracking line search: the next iterate is computed using the retraction (\ref{eq:retraction}) such that ${ x}_{+} = R_{{x}}( {s}{\xi}_x)$, where the step-size $s$ is obtained with backtracking line search.
\item Repeat until convergence.
\end{enumerate}
  \\
 \hline
\end{tabular}
\end{center} 
\caption{The Riemannian steepest-descent algorithm.}
\label{tab:Riemannian_steepest_descent} 
\end{table}

\subsubsection*{Riemannian gradient} The horizontal lift of the Riemannian gradient $\grad_{[x]}$ of a cost function, say $f : \mathcal{M} \rightarrow \mathbb{R}$, on the quotient manifold $\mathcal{M}/\sim$ is uniquely represented by the matrix representation, i.e., the
\begin{equation}\label{eq:Riemannian_gradient}
 \change{\hlift} { {\grad}_{[x]} f} = \grad_{ {x}}  {f},
\end{equation}
where $\grad_{ {x}}  {f}$ is the gradient of $f$ on the computational space ${\mathcal{M}}$ at $x$. \change{The equality in (\ref{eq:Riemannian_gradient}) is possible due to invariance of the cost function along the equivalence class $[ {x}]$, the choice of the Riemannian metric (\ref{eq:metric_quotient}), and the choice of the horizontal space $\mathcal{H}_x$ as the orthogonal complement of the vertical space $\mathcal{V}_x$ \cite[Section~3.6.2]{absil08a}.} 

The gradient $\grad_{ {x}}  {f}$ on the computational space is computed from its definition: it is the unique element of $T_{  x}  {\mathcal M}$ that satisfies \cite[Equation~3.31]{absil08a}
\begin{equation}\label{eq:Riemannian_gradient_computation}
{\rm D}  {f}(x)[ {\eta}_{  x}] =  {g}_{ {x}}(     \grad_{  x}  {f},    {\eta}_{ {x}}  ) \quad {\rm for\ all} \   {\eta}_{  x} \in T_{  x}  {\mathcal M},
\end{equation}
where $ {g}_{  x}$ is the Riemannian metric (\ref{eq:metric_quotient}) and ${\rm D}  {f}(x)[ {\eta}_{  x}]$ is the Euclidean directional derivative of $ {f}$ in the direction $\eta_{  x}$, i.e., ${\rm D}  {f}(x) [ {\eta}_{  x}] = \lim_{t \rightarrow 0} ( {f} ({ {x} + t  {\eta}_{  x} }) -  {f} ({  x}))/{t}$ \citep[Section~3.6]{absil08a}. An equivalent way of computing $\grad_{ {x}}  {f}$ is by solving the \emph{convex quadratic} problem
\begin{equation}\label{eq:Riemannian_gradient_computation_qp}
\begin{array}{lll}
\grad_{ {x}}  {f} = \argmin\limits_{\zeta_x \in T_x \mathcal{M}} f(x) - \langle f_x(x), \zeta_x \rangle+ \frac{1}{2} g_x(\zeta_x, \zeta_x),
\end{array}
\end{equation}
where $f_x(x)$ is the first-order derivative of the cost function $f$, $\langle \cdot, \cdot \rangle$ is the Euclidean inner product, and $ {g}_{  x}$ is the Riemannian metric (\ref{eq:metric_quotient}) at $x \in \mathcal{M}$. It should be noted that $\langle f_x(x), \zeta_x \rangle = {\rm D}  {f}(x)[ {\zeta}_{  x}]$. The equivalence between solutions to (\ref{eq:Riemannian_gradient_computation}) and (\ref{eq:Riemannian_gradient_computation_qp}) is established by observing that the condition (\ref{eq:Riemannian_gradient_computation}) is, in fact, equivalent to the optimality conditions of the convex quadratic problem (\ref{eq:Riemannian_gradient_computation_qp}).

\subsubsection*{Retraction} An iterative optimization algorithm involves computing a search direction and then ``moving in that direction.'' The default option on a Riemannian manifold is to move along geodesics, leading to the definition of the exponential map \citep[Section~5.4]{absil08a}. Because the calculation of the exponential map can be computationally  demanding, it is customary in the context of manifold optimization to relax the constraint of moving along geodesics. The exponential map is then relaxed to a {\it retraction}, which is any map $ {R}_{  x}: \mathcal{H}_{  x}  \rightarrow  {\mathcal{M}} $ that locally approximates the exponential map, up to first-order, on the manifold \citep[Definition~4.1.1]{absil08a}.
A natural update on the manifold is, thus, based on the update formula
\begin{equation}\label{eq:retraction}
	 {x}_{+} =  {R}_{ {x}}( s {\xi}_{ {x}}),
\end{equation}
where $ {\xi}_{  x} \in \mathcal{H}_{  x}$ is a search direction, \changeBM{$s$ is the step-size}, and $ {x}_+ \in  {\mathcal M}$.

\change{The retraction $ {R}_{  x}$ defines a retraction ${R}_{[x]}( {\xi}_{[x]}) : =[R_x (\xi_x)] $ on the Riemannian quotient manifold ${\mathcal{M}}/\sim$, provided that the equivalence class $[R_x (\xi_x)] $ does not depend on the specific choice of the matrix representations of $[x]$ and ${\xi}_{[x]}$. Here $ {\xi}_{  x}$ is the horizontal lift of an abstract tangent vector $\xi_{[x]} \in T_{[x]} (\mathcal{M} /\sim)$ in $\mathcal{H}_{  x }$ and $[\cdot]$ is the equivalence class defined earlier in the section. Equivalently, the retraction operation ${R}_{[x]}( {\xi}_{[x]}) : =[R_x (\xi_x)] $ is well defined on $\mathcal{M}/\sim$ when $R_x (\xi_x)$ and $R_y (\xi_y)$ belong to the same equivalence class, i.e., $[R_x (\xi_x)] = [R_y (\xi_y)]$ for all $x, y \in [x]$.}

\subsection{Inferring a metric from the Lagrangian}\label{sec:connection}
The practical performance of the Riemannian steepest-descent algorithm in Table \ref{tab:Riemannian_steepest_descent} greatly depends on the choice of the metric (\ref{eq:metric_quotient}). The dominant trend in Riemannian optimization has been to infer the metric from the geometry of the search space. Symmetry properties of the search space suggest choosing invariant metrics, that is, metrics that are not affected by a symmetry transformation of variables. In many situations, invariance properties of search space single out the choice of a unique metric \citep{absil08a, edelman98a, moakher02a}. However, the metrics that are motivated solely by the search space may not perform well in an optimization setup as they do not take into consideration the cost function \citep{manton02a}.

\change{To address the above issue, we connect the SQP approach in Table \ref{tab:sqp} to the Riemannian steepest-descent algorithm in Table \ref{tab:Riemannian_steepest_descent} with a specific metric that is induced by the Lagrangian (\ref{eq:Lagrangian}). In particular, we extend the ideas of Propositions \ref{prop:connecting_SQP_Newton} and \ref{prop:second_order_SQP_Newton} to the case of quotient manifolds. The extension has a twofold objective. First, it provides a guidance in choosing metrics on a manifold with symmetries. Second, it provides a framework to extend SQP to constraints with symmetries. }


\begin{theorem}\label{thm:connection}
\changeBMM{Consider an equivalence relation $\sim$ in the space $ {\mathcal{M}}$.} \changeBM{Assume that both $\mathcal{M}$ and $\mathcal{M}/\sim$ have the structure of a Riemannian manifold with the metric $\bar{g}$ on $\mathcal{M}$ and that the function $f: \mathcal{M} \rightarrow \mathbb{R}$ is a smooth function with isolated minima on the quotient manifold $\mathcal{M}/\sim$.} Assume also that $\mathcal{M}$ has the structure of an embedded submanifold in $\mathbb{R}^n$. 

If $x^* \in \mathcal{M}$ is a local minimum of $f:\mathcal{M} \rightarrow \mathbb{R}$ on $\mathcal{M}$, then the following hold:
\begin{enumerate}[(i)]
\item \change{$\langle \eta_{x^*} , \D^2 {\mathcal{L}}( {x}^*,  {\lambda}_{  x^*}) [ {\eta}_{  x^*}]\rangle = 0 $ for all $ \eta_{x^*} \in \mathcal{V}_{x^*}$, and}

\item \change{the quantity $\langle \xi_{x^*} , \D^2 {\mathcal{L}}( {x}^*,  {\lambda}_{  x^*}) [ {\xi}_{  x^*}]\rangle$ captures \changeBM{all} second-order information of the cost function $f$ on $\mathcal{M}/\sim$ for all $ \xi_{x^*} \in \mathcal{H}_{x^*}$,}
\end{enumerate}
where $\mathcal{V}_{x^*}$ is the vertical space, $\mathcal{H}_{x^*}$ is the horizontal space (the subspace of $T_{x^*} \mathcal{M}$ that is complementary to $\mathcal{V}_{x^*}$), $\langle \cdot, \cdot \rangle$ is the Euclidean inner product, and $\D^2 {\mathcal{L}}( {x^*},  {\lambda}_{  x^*}) [ {\xi}_{  x^*}]$ is the second-order derivative of $\mathcal{L} (x, \lambda_x)$ with respect to $x$ at $x^* \in \mathcal{M}$ applied in the direction $\xi_{x^*} \in \mathcal{H}_{x^*}$ and keeping $ {\lambda}_{  x^*}$ fixed to its least-squares estimate (\ref{eq:least_square_lambda}).

\end{theorem}

\begin{proof}

Statement (i) corresponds to the degeneracy of the Lagrangian at a local minimum and statement (ii) is an extension of Proposition \ref{prop:second_order_SQP_Newton} to the case of quotient manifolds. 


The Lagrangian \change{$\mathcal{L}(x^*, \lambda)$} at $x^*$, because both the cost and the constraint terms remain invariant under the equivalence \change{relation} $\sim$, is \emph{constant} along the equivalence class $[x^*] := \{ y^* \in \mathcal{M} : y^* \sim x^*\}$. It should be noted that since $x^*$ is a local minimum, the first-order derivative of the Lagrangian with respect to $x$ at $x^*$ is zero, i.e., \changeBMM{$\mathcal{L}_{x} (x^*, \lambda_{x^*}) = 0$}. \changeBM{It should be noted that $\mathcal{L}_{x} (x, \lambda_{x})$ is the derivative of $\mathcal{L}(x, \lambda_x)$ with respect to the first argument, i.e., $\lambda_x$ is kept fixed.} Differentiating the Lagrangian $\mathcal{L}(x, \lambda_x)$ twice at $x^*$ along the linearization of the equivalence class $[x^*]$, that is the vertical space $\mathcal{V}_{x^*}$, results in the first equality in (i).

To prove (ii) of the theorem, we \changeBM{exploit the} Riemannian structure on $\mathcal{M}/\sim$. The theory of Riemannian submersion \citep[Chapter~3]{absil08a} states that the horizontal space $\mathcal{H}_{x}$ is \changeBMM{\emph{orthogonal}} to the vertical space $\mathcal{V}_x$ with \changeBMM{respect to} the metric $\bar{g}_x$ and allows us to compute the Riemannian gradient and Hessian of $f$ on $\mathcal{M}/\sim$ using the \emph{orthogonal projection} of their counterparts in the total space $\mathcal{M}$. From the computation of the Riemannian gradient $\grad_x f$ and the first-order derivative of the Lagrangian we have the equality \citep[Equation~3.31]{absil08a} 
\[
\begin{array}{llll}
&{\rm D}  {f}(x)[ {\xi}_{  x}] = \langle \mathcal{L}_x (x, \lambda_{x}), {\xi}_{  x} \rangle = {\bar g}_{ {x}}(     \grad_{  x}  {f},    {\xi}_{ {x}}  ) \quad {\rm for\ all\ }\xi_x \in T_x \mathcal{M}.
\end{array}
\]
Taking the Euclidean directional derivative of the above equation \change{at} $x^*$ along $\xi_{x^*} \in \mathcal{H}_{x^*}$ with the additional information that $ \grad_{x^*} f= 0$ ($x^*$ is a local minimum),
\begin{equation}\label{eq:proof_eq}
\begin{array}{llll}
\langle \xi_{x^*},\D^2  {f}(x^*)[ {\xi}_{  x^*}] \rangle = \langle \xi_{x^*} , \D^2 {\mathcal{L}}( {x}^*,  {\lambda}_{  x^*}) [ {\xi}_{  x^*}]\rangle = \bar{g}_{x^*}(\xi_{x^*}, \D \grad_{x^*} f [\xi_{x^*}] ),
\end{array}
\end{equation}
where the specific operation $\D \grad_{x^*} f [\xi_{x^*}] $ should be treated as the \emph{Euclidean differential of a vector field}, the Riemannian gradient $\grad  _{x^*} f $, along $\xi_{x^*}$. \changeBM{In this case, $\D^2 {\mathcal{L}}( {x},  {\lambda}_{  x}) [ {\xi}_{  x}]$ is the second-order derivative of $\mathcal{L}(x, \lambda_x)$ with respect to $x$ along $\xi_x$ keeping $\lambda_x$ fixed.}

We define $\Pi_{x}:\mathbb{R}^n \rightarrow \mathcal{H}_x$ as the orthogonal projection operator in the metric $\bar{g}_x $, and ${\hess}_{x} f [\xi_{x}]$ as the Riemannian Hessian in the total space $\mathcal{M}$ applied along the direction $\xi_x \in \mathcal{H}_x$. The horizontal lift of the Riemannian Hessian ${\hess}_{[x]} f [\xi_{[x]}] $ on the quotient manifold $\mathcal{M}/\sim$ has the characterization $\Pi_{x}({\hess}_{x} f [\xi_{x}]) $ \citep[Proposition~5.3.3]{absil08a}. Additionally, the \emph{Koszul formula} relates the operation $\D \grad_{x} f [\xi_{x}] $ and the Riemannian Hessian in the total space $\mathcal{M}$ \citep[Theorem~5.3.1]{absil08a}. 

At the minimum $x^*$, second-order information of $f$ on the manifold $\mathcal{M}/\sim$ is captured by the quantity $\bar{g}_{x^*}(\xi_{x^*}, \Pi_{x^*}({\hess}_{x^*} f [\xi_{x^*}]) )$ for all $\xi_{x^*} \in \mathcal{H}_{x ^*}$. The term simplifies as follows:
\[
\begin{array}{lll}
\bar{g}_{x^*}(\xi_{x^*}, \underbrace{\Pi_{x^*}({\hess}_{x^*} f [\xi_{x^*}])}_{
\left \{\footnotesize
\begin{array}{ll}
\hlift {\rm the}\\
{\rm Hessian\ on\ } \mathcal{M}/\sim
\end{array}
\right \}
}) &= &\bar{g}_{x^*}(\xi_{x^*}, \underbrace{\displaystyle {\hess}_{x^*} f [\xi_{x^*}]}_{{\rm the\ Hessian\ on\ } \mathcal{M}}) \quad {\rm for\ } \Pi_{x^*}(\xi_{x^*}) =\xi_{x^*}\\
& = & \bar{g}_{x^*}(\xi_{x^*}, \D \grad_{x^*} f [\xi_{x^*}] ) \ \ {\rm from\ the\ Koszul\ formula\ } \\
&  &\ \ \quad \quad \qquad \qquad \qquad \qquad \ {\rm and\ } \grad_{x^*} f= 0 \\
& =& \langle \xi_{x^*} , \D^2 {\mathcal{L}}( {x}^*,  {\lambda}_{  x^*}) [ {\xi}_{  x^*}]\rangle
\quad  {\rm from\ } (\ref{eq:proof_eq}).
\end{array}
\]
This proves statement (ii). 

\end{proof}

\change{Theorem \ref{thm:connection} states that the underlying symmetries make the Lagrangian \emph{degenerate} in the tangent space $T_x \mathcal{M}$ of the total space $\mathcal{M}$. The other important observation is that the quantity $\langle \xi_x, \D^2 \mathcal{L} (x, \lambda_x)[\xi_x]\rangle$ captures second-order information at a local minimum along the horizontal space, where the horizontal space is \changeBMM{the} subspace of $T_x \mathcal{M}$ that is the \changeBMM{orthogonal complement} of the vertical space $\mathcal{V}_x$.}


As an immediate consequence of Theorem \ref{thm:connection}, we observe that in the neighborhood of a minimum, a selection of the search direction is given by solving 
\begin{equation}\label{eq:sqp_horizontal}
\begin{array}{lll}
\argmin\limits_{ \zeta_x \in \mathcal{H}_x} &  {f}(  x) + \langle  {f}_x (  x),  {\zeta}_{  x} \rangle + \frac{1}{2}\langle  {\zeta}_{  x}, \D^2 {\mathcal{L}}( {x},  {\lambda}_{  x}) [ {\zeta}_{  x}]\rangle, \\
\end{array}
\end{equation}
where \change{$\mathcal{H}_x$ is any predefined horizontal space characterization}. Followed by a retraction operation (Section \ref{sec:Riemannian_framework}), (\ref{eq:sqp_horizontal}) defines locally a steepest-descent algorithm on the quotient manifold $\mathcal{M}/\sim$, that is, an algorithm that iterates on the classes of equivalences. Here $\langle \cdot, \cdot\rangle$ is the Euclidean inner product and $f_x(x)$ is the first-order derivative of the function $f$. \change{The scheme with (\ref{eq:sqp_horizontal}) has the interpretation of a steepest-descent algorithm on the quotient manifold $\mathcal{M}/\sim$ and can be considered as an extension of the classical SQP approach, that has been defined on embedded constraints, to quotient manifolds. The significance of the result is twofold: first,  the computational problem (\ref{eq:sqp_horizontal}) is considerably simpler than the computational machinery needed for a steepest-descent or a quasi-Newton algorithm on a general quotient manifold. Second, it decouples the second-order information term, i.e., $\D^2 {\mathcal{L}}( {x},  {\lambda}_{  x})$, from the horizontal space characterization $\mathcal{H}_x$ that could be predefined.} \changeBM{The expression (\ref{eq:sqp_horizontal}) is appealing  as a starting point to equip the quotient manifold with a Riemannian structure since it captures second-order information  in the neighborhood of a minimum.}

\subsection{Regularized Lagrangians for Riemannian optimization}\label{sec:convexity}
\changeBM{Motivated by the expression (\ref{eq:sqp_horizontal}), we introduce the family of regularized Lagrangians
\begin{equation}\label{eq:regularized_Lagrangian}
 \mathcal{\tilde L}(x, \lambda_x) = \omega_1 f(x) - \omega_2 \langle \lambda_x, h(x) \rangle,
 \end{equation}
with the corresponding candidate metrics
\begin{equation}\label{eq:metric_decomposition}
\begin{array}{lll}
& g_x(\xi_x, \eta_x) & = \langle\xi_x, \D ^2 \mathcal{\tilde L} (x, \lambda_x) [\eta_x]\rangle \\
&               & = \omega_1 \underbrace{\langle\xi_x, \D^2 f (x) [\eta_x]\rangle}_{\rm cost\ related} + \omega_2 \underbrace{\langle\xi_x, \D^2 c (x, \lambda_x) [\eta_x]\rangle}_{\rm constraint\ related} 
\end{array}
\end{equation}
in the total space $\mathcal{M}$, \changeBMM{where $c(x, \lambda_x) = -\langle \lambda_x, h(x) \rangle $}. The choice of the regularizing parameters  $\omega_1 \in [0,1]$ and $\omega_2 \in [0,1]$  is problem dependent, but concrete choices  are discussed below and also illustrated in Figure \ref{fig:tradeoff}. The default choice of $\omega_1 = 1$ and $\omega_2 = 1$ is appealing because it captures the full second-order information. However, it also leads to a degenerate Lagrangian at the minimum. The choices of $\omega_1$ and $\omega_2$ thus result from a trade-off between ensuring the positive definiteness  of the metric candidate (\ref{eq:metric_decomposition}) in the tangent space $T_x \mathcal{M}$ while capturing second-order information near minima. }

\changeBM{If in addition to being positive definite the metric (\ref{eq:metric_decomposition}) is invariant along the equivalence class $[x]$,  and if the horizontal space $\mathcal{H}_x$ is  chosen as the orthogonal subspace to the vertical space $\mathcal{V}_x$, then the manifold $\mathcal{M}/\sim$ has the structure of a Riemannian submersion \citep{smith94a, edelman98a, absil08a}. This simplifies the computation of the search direction, which then amounts to solving the problem 
\begin{equation}\label{eq:sqp_mod}
\begin{array}{lll}
\argmin\limits_{ \zeta_x \in T_x \mathcal{M}} &  {f}(  x) + \langle  {f}_x (  x),  {\zeta}_{  x} \rangle + \frac{1}{2} g_x(\zeta_x, \zeta_x),
\end{array}
\end{equation} 
where $\langle \cdot, \cdot\rangle$ is the Euclidean inner product and $f_x(x)$ is the first-order derivative of the function $f$, and $g_x( \cdot, \cdot)$ is the Riemannian metric  (\ref{eq:metric_decomposition}). It should be noted that even though the minimization (\ref{eq:sqp_mod}) is in the total tangent space $T_x\mathcal{M}$, the solution $\zeta_x ^*$ of (\ref{eq:sqp_mod}), by construction, belongs to the chosen horizontal space $\mathcal{H}_x$ that is orthogonal to $\mathcal{V}_x$. }

We further discuss two scenarios below that suggest how to exploit the available structure to construct novel Riemannian metrics. The problem structure can be exploited in more general situations along the same lines.



\begin{figure}[t]
\centering
\includegraphics[scale = 0.50]{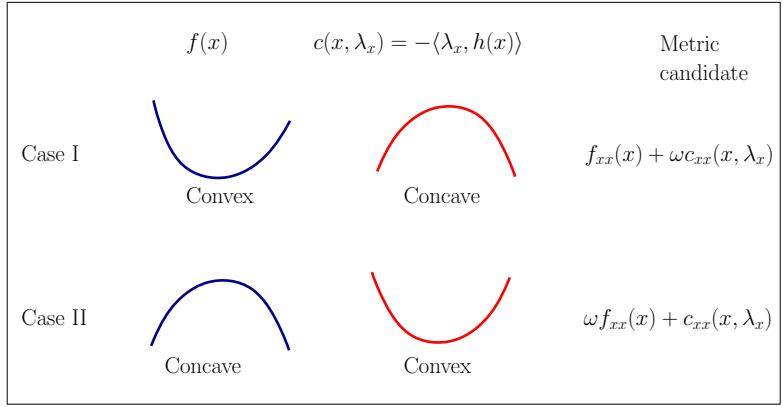}
\caption{Choosing metrics for the Riemannian steepest-descent algorithm in Table \ref{tab:Riemannian_steepest_descent}. Shown are two extreme situations in which the regularized Lagrangian (\ref{eq:regularized_Lagrangian}) provides a clear metric candidate locally in the neighborhood of a minimum. $f_{xx}(x)$ is the second-order derivative of $f(x)$ and $c_{xx}(x, \lambda_x)$ is the second-order derivative of $c(x, \lambda_x)$ with respect to $x$ keeping $\lambda_x$ fixed. Because of local convexity of the Lagrangian (on the tangent space) at a minimum, convex and concave structures of the function $f$ lead to \changeBM{a} well defined family of metrics parameterized by the regularization parameter $\omega \in [0, 1)$. It locally captures second-order information of the problem. The case $\omega = 1 $ is discarded to prevent singularity of the metric candidate at a local minimum. Additionally, to extend the metrics away from a minimum, the parameter $\omega$ is updated at every iteration, e.g., with the procedure shown in (\ref{eq:update_omega}).}
\label{fig:tradeoff}
\end{figure}

\subsubsection*{Case I: minimizing a strictly convex function} \label{sec:convex}
Consider the case when $ {f}$ is a \emph{strictly} convex function. In this case the second-order derivative $ {f}_{xx}(  x) \succ 0$ (due to strict convexity assumption) is a good metric candidate. In addition, locally in the neighborhood of a minimum, a family of Riemannian metrics is identified from (\ref{eq:metric_decomposition}) as 
\begin{equation}\label{eq:preconditioned_metric_quotient_convex}
\begin{array}{lll}
g_x(\xi_x, \eta_x) = \underbrace{\langle\xi_x,  \D^2f (x) [\eta_x]\rangle}_{f_{xx} \succ 0 {\rm\ and\ dominating}} + \omega \langle\xi_x, \D^2 c (x, \lambda_x) [\eta_x]\rangle 
\end{array}
\end{equation}
by selecting $\omega_1 = 1$ and $\omega_2 = \omega $, where $\omega \in [0, 1)$, $\xi_x, \eta_x$ are tangent vectors in $T_x \mathcal{M}$, $c(x, \lambda_x) = -\langle \lambda_x, h(x)\rangle$, and $\D^2 c (x, \lambda_x) [\eta_x]$ is the second-order derivative of $c(x, \lambda_x)$ with respect to $x$ keeping $\lambda_x$ fixed and that is applied along the direction $\eta_x$. \change{The case $\omega = 1 $ is discarded to prevent singularity of the metric candidate (\ref{eq:preconditioned_metric_quotient_convex}) in $T_x \mathcal{M}$ in the neighborhood of a minimum.}

\subsubsection*{Case II: maximizing a strictly convex function} \label{sec:concave}
Consider the problem of maximizing a convex cost function, that is equivalent to minimizing a concave cost function, on a manifold. In this case, $ {f}_{xx}(x, \lambda_x) \prec 0$, and locally in the neighborhood of a minimum, second-order information of $c(x, \lambda_x) = -\langle \lambda_x, h(x)\rangle$ is the source of convexity. This fact follows from \changeBMM{the} second-order optimality condition of the optimization problem \citep[Chapter~18]{nocedal06a}. Here the problem structure suggests a family of Riemannian metrics
\begin{equation}\label{eq:preconditioned_metric_quotient_concave}
\begin{array}{lll}
 g_x(\xi_x, \eta_x)  = \omega \langle\xi_x, \D^2  f (x) [\eta_x]\rangle + \underbrace{\langle\xi_x,  \D ^2 c (x, \lambda_x) [\eta_x]\rangle}_{c_{xx} {\rm\ is\ locally\ positive\ definite}} 
\end{array}
\end{equation}
by selecting $\omega_1 = \omega$ and $\omega_2 = 1$ in (\ref{eq:metric_decomposition}), where $\omega \in [0, 1)$, $\xi_x, \eta_x$ are tangent vectors in $T_x \mathcal{M}$, $c(x) = -\langle \lambda_x, h(x)\rangle$, $c_{xx}(x, \lambda_x)$ is the second-order derivative of $c(x, \lambda_x)$ with respect to $x$ keeping $\lambda_x$ fixed. \change{Once again $\omega = 1 $ is discarded to prevent singularity of the metric candidate (\ref{eq:preconditioned_metric_quotient_concave}) in $T_x \mathcal{M}$ in the neighborhood of a minimum.}

\subsubsection*{Globalizing the local metrics} \label{sec:weight}
The parameter $\omega \in [0, 1)$ in the metrics (\ref{eq:preconditioned_metric_quotient_convex}) and (\ref{eq:preconditioned_metric_quotient_concave}), apart from providing a family of Riemannian metrics, also plays a critical role in the numerical performance of the Riemannian steepest-descent algorithm in Table \ref{tab:Riemannian_steepest_descent}. With $\omega = 0$, the Riemannian metric captures only part of second-order information and therefore, locally in the neighborhood of a minimum, the Riemannian steepest-descent algorithm may converge poorly, e.g., linearly. On the other hand with $\omega \to 1$, the Riemannian metric captures the full second-order information and the Riemannian steepest-descent algorithm is expected to show better convergence. A numerical technique for interpolating between these two extreme scenarios is to vary $\omega = [0, 1)$ at every iteration with \changeBM{an} increasing barrier function that tends to $1$ as the number of iterations increases. A simple updating technique is $\omega(k) =  1 - 2^{k-1}$, where $k$ is the iteration number. A strategy to safeguard against a non-descent search direction, e.g., \changeBMM{when solving the quadratic programming problem (\ref{eq:sqp_mod}) with the metric (\ref{eq:metric_decomposition})}, is to ignore the updated $\omega$ that resulted in a non-descent direction (checking this numerically is straightforward) and restart the procedure of updating $\omega$ again.

A different technique is to modify $\omega$ as and when required. For example, defining $\delta = 1 - \omega$, we have the strategy where at the $k{\rm th}$ iteration
\begin{equation}\label{eq:update_omega}
\delta_{k} = \left \{
\begin{array}{lll}
0.5\delta_{k-1} &  {\rm when\ a\ descent\ direction\ is\ obtained}\\
4\delta_{k -1} &  {\rm when\ a\ non-descent\ direction\ is\ obtained},
\end{array}
\right .
\end{equation}
with $\delta_0 = 1$. Care is taken to ensure that $\omega \in [0, 1)$ for all iterations.

Safeguards similar to the trust-regions, i.e., by constraining the norm of the search direction, can also be implemented to ensure that the search direction computed with the Riemannian metric remains a locally descent direction \citep[Section~18.5]{nocedal06a}.


\section{Quadratic optimization with orthogonality constraints: revisiting the generalized eigenvalue problem}\label{sec:quadratic_orthogonality}

Constrained quadratic optimization problems arise naturally in a number of applications, especially while solving linear systems of matrix equations \citep[Section~2.2]{absil08a}. Also popular are the \emph{orthogonality} constraints in large-scale problems that are imposed to identify relevant smaller dimensional subspaces \citep{edelman98a}. Specific optimization problems include the generalized eigenvalue problem \citep{edelman98a, absil02a, absil10a}, the generalized orthogonal Procrustes problem \citep{elden99a}, and the joint diagonalization problem in signal processing \citep{theis09a}, to name just a few. 

For the sake of illustration, we specifically focus on the well-studied generalized eigenvalue problem that computes the smallest eigenvalues and eigenvectors of the matrix $\mat{B}^{-1}\mat{A}$, where $\mat{A}$ is a symmetric matrix of size $n\times n$ and $\mat{B}$ is a symmetric positive definite matrix of size $n\times n$ [\citealp[Section~4.5]{edelman98a}, \citealp[Chapter~8]{golub96a}]. This is realized by solving the optimization problem below iteratively, an extensively researched question in the literature [\citealp[Chapter~8]{golub96a}, \citealp{absil04c}]. In this section we exploit the quadratic nature of the cost function and the constraints to show that a family of Riemannian metrics has a simple characterization. It is also shown that the algorithms that result from the proposed metrics connect to a number of established algorithms, \change{each} of which is interpreted as a steepest-descent algorithm with a specific Riemannian metric \change{choice}.

The minimal $r$-eigenspace of $\mat{B}^{-1}\mat{A}$ is computed iteratively by solving the constrained quadratic optimization problem
\begin{equation}\label{eq:generalized_eigenvalue_problem}
\begin{array}{lll}
\min\limits_{{\mat X}\in \mathbb{R}^{n\times r}}  & \frac{1}{2} \trace( \mat{X}^T\mat{AX}  )\\
\subject & \mat{X}^T\mat{BX} = \mat{I},
\end{array}
\end{equation}
where the constraint set of $n\times r$ matrices that satisfy $\mat{X}^T\mat{BX} = \mat{I}$ is known as the \emph{generalized Stiefel manifold} $\Gstiefel{r}{n}{\mat B}$. The constraint enforces orthogonality among vectors in coordinates spanned by $\mat{B} ^{1/2}$. Specifically when $\mat{B} = \mat{I}$, the generalized Stiefel manifold is the popular Stiefel manifold $\Stiefel{r}{n} := \{\mat{X} \in \mathbb{R}^{n \times r}: \mat{X}^T\mat{X} = \mat{I} \}$ \citep{edelman98a}.

\changeBM{It should be noted that the optimization problem (\ref{eq:generalized_eigenvalue_problem}) is invariant under the transformation $\mat{X} \mapsto\mat{XO}$ for all $\mat O \in \OG{r}$, where $ \OG{r}$ is the set of $r\times r$ orthogonal matrices.} As a consequence, the problem (\ref{eq:generalized_eigenvalue_problem}) is an optimization problem on the abstract quotient space $\Gstiefel{r}{n}{\mat B}/ \OG{r}$, also known as the \emph{generalized Grassmann manifold}. For the case $\mat{B} = \mat{I}$, this again boils down to the well-known Grassmann manifold $\Grass{r}{n}$, which is the set of $r$-dimensional subspaces in $\mathbb{R}^n$ \citep{edelman98a}. The optimization problem (\ref{eq:generalized_eigenvalue_problem}) is, therefore, reformulated on the generalized Grassmann quotient manifold, i.e.,
\begin{equation}\label{eq:grassmann_optimization}
\begin{array}{lll}
\min\limits_{\mat{X} \in \mathbb{R}^{n\times r}}  & \frac{1}{2}\trace(\mat{X}^T\mat{AX})\\
\subject & [{\mat X}] \in \Gstiefel{r}{n}{\mat B}/ \OG{r},
\end{array}
\end{equation}
where the optimization is on the set of equivalence classes $[{\mat X}]: = \{ \mat{XO}: \mat{O} \in \OG{r} \}$ at $\mat{X} \in \Gstiefel{r}{n}{\mat B}$.

The \change{\emph{conventional} choice of the metric} in the Riemannian framework is
\begin{equation}\label{eq:grassmann_metric_standard}
 {g}_{  x}( {\eta}_{  x},  {\xi}_{  x} )  = \trace(  {\eta}_{  x} ^T  {\xi}_{  x}  ),
\end{equation} 
where $ {x} = \mat{X} \in \Gstiefel{r}{n}{\mat B}$ and $ {\xi}_{  x},  {\eta}_{  x}$ are vectors in the tangent space of the constraints (the matrix characterization \change{of} the tangent space is shown in Table \ref{tab:grassmann}). \change{It is the unique metric that is motivated by \emph{invariance} on the orthogonal group, i.e., for the case when $r = n$ it is the unique choice \citep{moakher02a}. Hence, it becomes a conventional choice for $r < n$.} Because of its simplicity and its geometric consideration, the metric (\ref{eq:grassmann_metric_standard}) is also advocated by \citet{edelman98a, absil08a}. 
 
In contrast, the developments in Section \ref{sec:convexity} suggest a \emph{family of Riemannian metrics} that take the complete problem structure into account by computing the regularized Lagrangian (\ref{eq:regularized_Lagrangian}) and its second-order derivative. To this end, we have the following matrix representations for (\ref{eq:generalized_eigenvalue_problem}).
\begin{equation}\label{eq:lagrangian_grassmann}
\begin{array}{lrl}
&  {\mathcal L}( {x},  {\lambda}_{  x}) = & \trace(\mat{X}^T\mat{AX}) / 2 -  \langle  {\lambda}_{  x},\mat{X}^T \mat{BX} - \mat{I} \rangle / 2 \\
\Rightarrow &  {\mathcal L}_{x}( {x},  {\lambda}_{  x}) =
& \mat{AX} - \mat{BX} {\lambda}_{  x} \\
\Rightarrow &  \D ^2{\mathcal L}( {x},  {\lambda}_{  x}) [\xi_{x}] =
& \mat{A}\xi_x - \mat{B}\xi_x {\lambda}_{  x}, \\
\end{array}
\end{equation}
where $x$ has the matrix representation $\mat{X} \in \Gstiefel{r}{n}{\mat B}$, $\langle \cdot, \cdot \rangle$ is the Euclidean inner product, and the least-squares Lagrange multiplier is $ {\lambda}_{  x}= \allowbreak \Sym(  \allowbreak (\mat{X}^T\mat{BBX})^{-1} \allowbreak (\mat{X}^T \mat{BAX}))$ from (\ref{eq:least_square_lambda}) with the additional symmetry condition from the constraint, where $\Sym(\cdot)$ extracts the symmetric part of a square matrix, i.e., $\Sym({\mat D}) = (\mat{D} + \mat{D}^T)/2$. $\mathcal{L}_x (x, \lambda_x)$ is the first-order derivative of $\mathcal{L}(x, \lambda_x)$ and $\D ^2 \mathcal{L}(x, \lambda_x)[\xi_x]$ is the second-order derivative of $\mathcal{L}(x, \lambda_x)$ applied in the direction $\xi_x$, both computed while keeping $\lambda_x$ fixed.

\change{{It should be noted that the least-squares estimate $ {\lambda}_{  x}= \Sym( (\mat{X}^T\mat{BBX})^{-1}\quad \allowbreak (\mat{X}^T \mat{BAX}))$ is the solution to the problem $\argmin_{\lambda \in \mathbb{R}^{r\times r}} \allowbreak \| \mat{AX} - \mat{BX}\lambda  \|_Q ^2$ such that $\lambda$ is symmetric, where $\| \mat{AX} - \mat{BX}\lambda  \|_Q ^2 = \trace((\mat{AX} - \mat{BX}\lambda)^T \mat{Q} (\mat{AX} - \mat{BX}\lambda) )$ and $\mat{Q} = \mat{BX} (\mat{X}^T\mat{BBX})^{-2} \mat{X}^T\mat{B}$. A different estimate of $\lambda_x$ is obtained by solving the problem $\argmin_{\lambda \in \mathbb{R}^{r\times r}} \| \mat{AX} - \mat{BX}\lambda  \|_F ^2$ such that $\lambda$ is symmetric. Both these estimates coincide at a local minimum \changeBM{and} can be used in the neighborhood of a minimum.}}

It is readily checked that the Lagrangian ${\mathcal L}( {x},  {\lambda}_{  x})$ in (\ref{eq:lagrangian_grassmann}) remains unchanged under the action $\mat{X} \mapsto \mat{XO}$ for all $\mat{O} \in \OG{r}$, \changeBM{where $\mat{X} \in \Gstiefel{r}{n}{\mat B}$}. The action $\mat{X} \mapsto \mat{XO}$ also leads to the transformation in the least-squares Lagrange multiplier as $\lambda_x \mapsto \mat{O}^T \lambda_x \mat{O}$. Subsequently, we have the following proposition for constructing a family of Riemannian metrics for the problem (\ref{eq:grassmann_optimization}) on the generalized Grassmann manifold.

\begin{proposition}\label{prop:grassmann}
\changeBMM{There exists a family of Riemannian metrics}
\begin{equation} \label{eq:grassmann_metric_proposition}
\begin{array}{lll}
g_x(\xi_x, \eta_x) = \omega_1\underbrace{\langle \xi_x,\mat{A}	 \eta_x\rangle}_{\rm cost\ related} -  \omega_2\underbrace{\langle \xi_x,\mat{B} \eta_x \lambda_x \rangle}_{\rm constraint\ related}
\end{array}
\end{equation}
\changeBMM{on $\Gstiefel{r}{n}{\mat B}$ with $\omega_1, \omega_2 \in [0, 1]$, each of which induces a Riemannian metric on the quotient manifold $\Gstiefel{r}{n}{\mat B}/ \OG{r}$ in the neighborhood of the minimum of the quadratic optimization problem (\ref{eq:grassmann_optimization}). The parameters $\omega_1$ and $ \omega_2$ weight the cost and constraint, respectively.} $\xi_x$ and $ \eta_x$ are vectors in the tangent space of the constraints at $x = \mat{X}$ such that $\mat{X}^T\mat{BX} = \mat{I}$ and $\lambda_x = \Sym((\mat{X}^T\mat{BBX})^{-1}\allowbreak( \mat{X}^T \mat{B} \mat{A}\mat{X}  ))$ is the least-squares estimate, where $\Sym(\cdot)$ extracts the symmetric part of a square matrix, i.e., $\Sym({\mat D}) \allowbreak= (\mat{D} + \mat{D}^T)/2$. 
\end{proposition}

\begin{proof}

First, we show that the metrics (\ref{eq:grassmann_metric_proposition}) respect the symmetry condition in (\ref{eq:metric_quotient}). In order to show that the metric does not change along the equivalence class $[x] = [{\mat X}] = \{ \mat{XO}: \mat{O} \in \OG{r} \}$ for all $\mat{O} \in \OG{r} $, it is equivalent, but simplified following \citep[Proposition~3.6.1]{absil08a},  to show that the metric for tangent vectors ${\xi}_{  x}$ and $ \eta_x$ does not change under the transformations $\mat{X} \mapsto \mat{XO}$, $\eta_{x}, \mapsto \eta_{x} \mat{O}$, and $\xi_{x} \mapsto \xi_{x} \mat{O}$. These \changeBM{transformations} lead to the transformation in the least-squares Lagrange multiplier as $\lambda_x \mapsto \mat{O}^T \lambda_x \mat{O}$. Finally, a few extra computations show that the metrics (\ref{eq:grassmann_metric_proposition}) indeed remain the same under the mentioned transformations.

Second, we show the construction of one particular family of Riemannian metrics. To this end, consider the case when $\mat{A} \succ 0$. Consider the parameters $\omega_1 = 1$ and $\omega_2 = \omega$. Restricting $\omega \in [0, 1)$ guarantees that in the neighborhood of a minimum $g_x(\zeta_x , \zeta_x) > 0$ for all $\zeta_x$ in the tangent space, satisfying the criterion of positive definiteness. 

The proofs for symmetry compatibility and positive definiteness on the tangent space \changeBMM{conclude} the proof of the proposition.

\end{proof}



\begin{table}[H]
\begin{center} \small
\begin{tabular}{ p{4.5cm} | p{8cm} }
&  
$
\begin{array}[t]{lll}
\begin{array}{lll}
\min\limits_{\mat{X} \in \mathbb{R}^{n\times r}}  & \trace(\mat{X}^T\mat{AX})  /2\\
\subject & \mat{X}^T \mat{BX} = \mat{I}
\end{array}
\end{array}
$
\\
\hline
& \\
$\begin{array}[t]{lll} {\rm Matrix\ representation}\\
{\rm of\ an\ element\ } x \in \mathcal{M} \end{array}$& $\begin{array}[t]{lll} x = \mat{X}\end{array}$\\
& \\
$\begin{array}[t]{lll} {\rm  Computational\  space\  }  {\mathcal M} \end{array}$&
$
\begin{array}[t]{ll}
\Gstiefel{r}{n}{\mat B} = \{\mat{X} \in \mathbb{R}^{n \times r}: \mat{X}^T\mat{BX} = \mat{I} \}
\end{array}
$
\\
& \\
$\begin{array}[t]{lll} {\rm Group\ action }\end{array}$ &
$\begin{array}[t]{lll}\mat{XO}, \end{array}\ \forall\mat{O} \in  \OG{r} {\rm \ such\ that\ } \mat{O}^T{\mat O} = \mat{OO}^T = \mat{I}$ 
\\
& \\
$\begin{array}[t]{lll} {\rm Quotient\ space\ }  \end{array}$ &  
$
\begin{array}[t]{lll}
\Gstiefel{r}{n}{\mat B} /  \OG{r} 
\end{array}
$
\\
& \\
$
\begin{array}[t]{ll}
{\rm Tangent\ vectors\ in\ } 
T_{  x} {\mathcal M}
\end{array}
$ & $
\begin{array}[t]{lll}
\{  {\xi}_{  x}  \in  
\mathbb{R}^{n \times r}:  {\xi}_{  x}^T\mat{BX} + \mat{X}^T \mat{B}  {\xi}_{  x} = 0  \}
\end{array}
$
 \\
 & \\
$
\begin{array}[t]{lll}
{\rm Metric \ } {g}_{  x} ( {\xi}_{  x},   {\zeta}_{  x}) \\
{\rm for \ }  {\xi}_{  x},  {\zeta}_{  x} \in T_{  x}  {\mathcal M}
\end{array}
$
& 
$
\begin{array}[t]{lll}
g_{  x} ( {\xi}_{  x},  {\zeta}_{  x} ) 
 =  \omega_1\trace( {\zeta}_{  x} ^T \mat{A}  {\xi}_{  x})  \\
\quad \quad \quad \quad \quad \quad  - \omega_2\trace( {\zeta}_{  x} ^T \mat{B}  {\xi}_{  x} \lambda_x) \\
\\
{\rm or\ the\ metrics\ proposed\ in\ Section\ } \ref{sec:metric_tuning_grassmann},\\
{\rm where\ }  {\lambda}_{  x} =  \Sym((\mat{X}^T\mat{BBX})^{-1}(\mat{X}^T \mat{BAX})) \\
\end{array}
$\\

& \\
$
\begin{array}[t]{lll}
 {\rm Cost\ function}
\end{array}
$

&
$
\begin{array}[t]{ll}
 {f}(  x) = \trace(\mat{X}^T\mat{AX})  /2
\end{array}
$
\\
& \\
$
\begin{array}[t]{lll}
{\rm First}\mbox{-}{\rm order\ derivative\ of }\\
 {f}(  x)
\end{array}
$

&
$
\begin{array}[t]{ll}
 {f}_x (  x) = \mat{AX}
\end{array}
$
\\
& \\
$
\begin{array}[t]{lll}
{\rm Search\ direction\ }
\end{array}
$

&
$
\begin{array}[t]{ll}
   \argmin\limits_{ {\zeta}_{  x} \in T_{  x} {\mathcal M}} \quad {f}(  x) + \langle  {f}_x(  x),  {\zeta}_{  x}\rangle +\frac{1}{2}  {g}_{  x}( {\zeta}_{  x},  {\zeta}_{  x})\\
\end{array}
$
\\
& \\
$
\begin{array}[t]{lll}
{\rm Retraction\ }  {R}_{  x}( {\xi}_{  x}){\rm \ that } \\

{\rm maps\ a\ search\ direction\ } {  \xi}_{  x}\\
{\rm onto\ }  {\mathcal M}
\end{array}
$
&
$
\begin{array}[t]{lll}
\mat{UV}^T,\\
\mbox{where }  \mat{X} + {\xi}_{  x} = \mat{U\Sigma V}^T \mbox{ is the restriction of thin SVD}\\ 
\mbox{such that } \mat{U}^T \mat{BU} = \mat{I},\mat{V}^T\mat{V} = \mat{I}, \mbox{ and }\\
 \mat{\Sigma} \mbox{ is a positive diagonal matrix.}
\end{array}
$
\\
 & \\
 \hline
\end{tabular}
\end{center} 
\caption{Optimization-related ingredients for computing the extreme eigenvalues of $\mat{B}^{-1} \mat{A}$. The numeric complexity per iteration depends on solving the quadratic programming problem for the search direction computation. In many instances exploiting sparsity in matrices $\mat{A}$ and $\mat B$ leads to numerically efficient schemes. Here $\Sym(\cdot)$ extracts the symmetric part of a square matrix, i.e., $\Sym({\mat D}) = (\mat{D} + \mat{D}^T)/2$. Few choices of the regularizing parameters $\omega_1, \omega_2 \in [0, 1]$ for relevant situations are shown in Section \ref{sec:metric_tuning_grassmann}.}
\label{tab:grassmann} 
\end{table}

\subsection{Metric tuning and shift policies}\label{sec:metric_tuning_grassmann}
Due to the quadratic nature of both cost and constraints, the metric (\ref{eq:grassmann_metric_proposition}) has the appealing feature of being parameterized by the Lagrangian parameter $\lambda_x$. This object is low dimensional when $r \ll n$. It provides an interesting interpretation of various ``shift'' policies developed in numerical linear algebra for eigenspace computations \citep[Chapter~8]{golub96a}. \change{We further particularize the selection of regularization parameters $\omega_1$ and $\omega_2$ in (\ref{eq:grassmann_metric_proposition}) when $\mat{A} \succ 0$ and when $\mat{A} \not \succ 0$.} In both these cases, we propose metrics \change{that} connect to a number of classical algorithms for the generalized eigenvalue problem \citep{absil02a, absil04c, absil10a}.


\subsubsection*{When $\mat{A} \succ 0$}
This instance falls under Case I of Section \ref{sec:weight}, i.e., with $\omega_1 = 1$ and $\omega_2 = \omega$ in (\ref{eq:metric_decomposition}). Therefore, the family of proposed Riemannian metrics has the structure
\begin{equation}\label{eq:inverse_iteration_metric}
\begin{array}{lr}
 {g}_{  x}( {\xi}_{  x},  {\zeta}_{  x}) = \trace( {\xi}_{  x}^T \mat{A}  {\zeta}_{  x}) - \omega  \trace( {\xi}_{  x} ^T \mat{B}  {\zeta}_{  x} \lambda_x), 
 \end{array}
\end{equation}
where $ {\xi}_{  x}$ and $ {\zeta}_{  x} $ are vectors in the tangent space of \changeBMM{the} constraints, the least-squares Lagrange multiplier $\lambda_x = \Sym((\mat{X}^T\mat{BBX})^{-1}(\mat{X}^T \mat{BAX}))$, and $\omega =[0, 1)$. 

The metric (\ref{eq:inverse_iteration_metric}) provides two insightful connections to the literature. First, the proposed metric (\ref{eq:inverse_iteration_metric}) with $\omega = 0$ generalizes the well-known \emph{inverse iteration algorithm} for computing the smallest eigenvalues of a symmetric matrix \citep[Section~8.2.2]{golub96a}. For the case when $\mat{B} = \mat{I}$, the \emph{negative} Riemannian gradient with the metric (\ref{eq:inverse_iteration_metric}) and $\omega = 0$ is computed as in Table \ref{tab:grassmann} as 
\[
\left .
\begin{array}{lll}
\argmin\limits_{\substack{\zeta_x \in \mathbb{R}^{n \times r} \\
\zeta_x ^T \mat{X} + \mat{X}^T \zeta_x = 0   } } \quad \langle \mat{AX}, \zeta_x \rangle + \frac{1}{2} \trace(\zeta_x ^T \mat{A} \zeta_x)
\end{array}
\right \} = \mat{A}^{-1}\mat{X}(\mat{X}^T \mat{A}^{-1} \mat{X})^{-1} - \mat{X}.
\]
Given an iterate $x = \mat{X}$ such that $\mat{X}^T \mat{X} = \mat{I}$, the Riemannian steepest-descent update with \emph{unit step-size}, is $ {x}_+  =  {R}_{  x}(\mat{A}^{-1}\mat{X}\allowbreak(\mat{X}^T \mat{A}^{-1} \mat{X})^{-1}\allowbreak - \mat{X})$, which is equivalent \changeBMM{to the update $ x_+ = {\rm qf}(\mat{A}^{-1}\mat{X})$} \changeBMM{on the Grassmann manifold}, where $R_x(\cdot)$ is the retraction operator defined in Table \ref{tab:grassmann} and ${\rm qf}(\mat{A}^{-1}\mat{X} )$ is the Q-factor of the QR decomposition of $\mat{A}^{-1}\mat{X}$. This is also the classical inverse iteration update \citep[Section~8.2.2]{golub96a}. This shows that the inverse iteration has the interpretation of a Riemannian steepest-descent algorithm with the metric (\ref{eq:indefinite_power_iteration_metric}) for $\omega = 0$.


A second insight is obtained for the case when $\omega$ is updated with iterations. \changeBM{The} Riemannian steepest-descent algorithm with the metric (\ref{eq:inverse_iteration_metric}) generalizes the popular Rayleigh quotient iteration algorithm (see \citealp[Section~8.2.3]{golub96a}, \citealp[Algorithm~4.2]{absil02a}, \citealp{absil04c}). \change{Given an iterate $x = \mat{X}$ such that $\mat{X}^T \mat{X} = \mat{I}$, \citet[Algorithm~4.2]{absil02a} propose the \emph{Grassmann-Rayleigh quotient iteration} (GRQI) algorithm for computing the next update $x_+$ with the steps
\begin{equation}\label{eq:absil_grqi}
\left.
\begin{array}{lll}
\mat{A}\mat{Z} - \mat{Z} \mat{X}^T \mat{AX} & =&  \mat{X} \quad {\rm (we\ solve\ for\ } \mat{Z}) \\
x_+   &=& {\rm qf}(\mat{Z} ).
\end{array}
\right \}
\end{equation}
}

To show that the Riemannian steepest-descent algorithm with the metric (\ref{eq:inverse_iteration_metric}) generalizes GRQI algorithm (\ref{eq:absil_grqi}), we consider again the case  when $\mat{B} = \mat{I}$. At each iteration of the Riemannian steepest-descent algorithm with the metric (\ref{eq:inverse_iteration_metric}), we are required to solve the system of linear equations (by looking at the optimality conditions of the quadratic program for computing the search direction) for $\zeta_{x} \in \mathbb{R}^{n \times r}$ and an auxiliary variable $\mu_{x} \in \mathbb{R}^{r\times r}$ \change{($\mu_x$ is the Lagrange multiplier corresponding to the constraint $\zeta_x ^T \mat{X} + \mat{X}^T \zeta_x = 0$)} of the form
\begin{equation}\label{eq:linear_system_grqi}
\begin{array}{llll}

& \left .
\begin{array}{lll}
\argmin\limits_{\substack{\zeta_x \in \mathbb{R}^{n \times r} \\
\zeta_x ^T \mat{X} + \mat{X}^T \zeta_x = 0   } } \quad \langle \mat{AX}, \zeta_x \rangle + \frac{1}{2} (\trace(\zeta_x ^T \mat{A} \zeta_x) - \omega  \trace( {\zeta}_{  x} ^T {\zeta}_{  x} \lambda_x))
\end{array}
\right \} \\

\\
\Rightarrow  & \left \{
\begin{array}{lll}
\mat{A} \zeta_x^* - \omega \zeta_x^* \lambda_x = \mat{X} \mu_x^* - \mat{AX}\\
\mat{X}^T \zeta_x^* +{ \zeta_x^*} ^T \mat{X} = 0,
\end{array}
\right.

\end{array}
\end{equation}
where $\omega \in [0, 1)$, \changeBMM{$\zeta_{x}^*$} is the computed search direction, and \changeBMM{$\mu_{x}^*$} is the \change{Lagrange multiplier} that guarantees that the search direction \changeBMM{$\zeta_{x}^*$} belongs to the tangent space of the constraints. It should be noted that the linear system of equations (\ref{eq:linear_system_grqi}) can be solved efficiently by exploiting additional sparsity structure in $\mat{A}$ \citep{absil02a}. Once the solution for \changeBM{\changeBMM{$\mu_{x}^*$} in (\ref{eq:linear_system_grqi}) is obtained by eliminating the constraint $\mat{X}^T \zeta_x^* +{ \zeta_x^*} ^T \mat{X} = 0$}, the Riemannian steepest-descent update $x_+$ with unit step-size is 
\begin{equation}\label{eq:update_grqi}
\begin{array}{llll}
& \left .
\begin{array}{lll}
\mat{A}\zeta_x^* - \omega \zeta_x^* \lambda_x  &= & \mat{X}\mu_x^* - \mat{AX} \\
x_+  & = &   R_x( \zeta_x^* ) \equiv  {\rm qf}(\mat{X} + \zeta_x^* ) \\
\end{array}
\right \} \\
\\

\Leftrightarrow & \left \{
\begin{array}{lll}
\mat{A}\mat{Z} - \omega\mat{Z} \lambda_x & =&  \mat{X}(\mu_x^* - \omega \lambda_x) \\
x_+   &=& {\rm qf}(\mat{Z} ),
\end{array}
\right .

\\
\\
\Leftrightarrow & \left \{
\begin{array}{lll}
\mat{A}\mat{Z} - \omega\mat{Z} \mat{X}^T\mat{AX} & =&  \mat{X}(\mu_x^* - \omega \lambda_x) \quad {\rm (we\ solve\ for\ } \mat{Z})\\
x_+   &=& {\rm qf}(\mat{Z} ),
\end{array}
\right .
\end{array}
\end{equation}
where $ \mat{Z} := \mat{X} + \zeta_x^* $, $R_x(\cdot)$ is the retraction operation defined in Table \ref{tab:grassmann}, and ${\rm qf}(\mat{Z})$ is the Q-factor of \change{the} QR decomposition of $\mat{Z}$. \changeBMM{It should be noted that $R_x( \zeta_x^* )$ and $ {\rm qf}(\mat{X} + \zeta_x^* )$ define the same element on the Grassmann manifold and hence the equivalence in (\ref{eq:update_grqi}). It should also be emphasized that the update (\ref{eq:update_grqi}) is similar to (\ref{eq:absil_grqi}) in the neighborhood of a minimum up to the extra terms $\omega$ and $\mu_x^* - \omega \lambda_x$. The similarity of (\ref{eq:update_grqi}) and (\ref{eq:absil_grqi}) is more profound for $r=1$ and when $\omega$ is updated, e.g., using (\ref{eq:update_omega}).} This connection also shows a way to extend GRQI beyond the neighborhood of a minimum. Consequently, GRQI update (\ref{eq:absil_grqi}) proposed by \citet{absil02a} has the interpretation of a Riemannian steepest-descent algorithm with the metric from \change{(\ref{eq:inverse_iteration_metric})}.


\subsubsection*{When $\mat{A} \not\succ 0$}
Consider first the case when $\mat{A} \prec 0$ that falls under Case II of Section \ref{sec:weight}, i.e., with $\omega_1 = \omega$ and $\omega_2 =1$ in (\ref{eq:metric_decomposition}), suggesting (locally) a family of Riemannian metrics \changeBM{that} has the form
\begin{equation}\label{eq:power_iteration_metric}
\begin{array}{lr}
 {g}_{  x}( {\xi}_{  x},  {\zeta}_{  x}) = \omega \trace( {\xi}_{  x}^T \mat{A}  {\zeta}_{  x}) -\trace( {\xi}_{  x}^T \mat{B}  {\zeta}_{  x}  {\lambda}_{  x}),\\
 \end{array}
\end{equation}
where $ {\xi}_{  x}$ and $ {\zeta}_{  x} $ are vectors in the tangent space of \changeBM{the} constraints and $\omega\in[0, 1)$. The expression for the least-squares Lagrange multiplier from (\ref{eq:least_square_lambda}) is $\lambda_x = \Sym((\mat{X}^T\mat{BBX})^{-1}\allowbreak(\mat{X}^T \mat{BAX}))$, where $\Sym(\cdot)$ extracts the symmetric part of a square matrix, i.e., $\Sym({\mat D}) = (\mat{D} + \mat{D}^T)/2$.

It should be noted that as $-\lambda_x$ is only guaranteed to be positive definite locally in the neighborhood of a local minimum, the metric characterization (\ref{eq:power_iteration_metric}) is a Riemannian metric \emph{only} in the neighborhood of a local minimum. In order to extend the metric away from a local minimum, we modify the metric, defined in (\ref{eq:power_iteration_metric}), by replacing $-\lambda_x$ with $( {\lambda}_{x}^T {\lambda}_{x})^{1/2}$ resulting in the modified metric
\begin{equation}\label{eq:indefinite_power_iteration_metric}
\begin{array}{lr}
 {g}_{  x}( {\xi}_{  x},  {\zeta}_{  x}) = \omega \trace( {\xi}_{  x}^T \mat{A}  {\zeta}_{  x}) +\trace( {\xi}_{  x}^T \mat{B}  {\zeta}_{  x} ( {\lambda}_{x}^T {\lambda}_{x})^{1/2}),
\end{array}
\end{equation}
where $( {\lambda}_{x}^T {\lambda}_{x})^{1/2}$ is the matrix square root of ${\lambda}_{x}^T {\lambda}_{x} $ that is \changeBMM{well defined} as long as $\lambda_x $ is full rank. \change{The full-rank assumption of $\lambda_x$ is required  for the metric, from (\ref{eq:indefinite_power_iteration_metric}), to be a smooth inner product.} The modified metric, shown in (\ref{eq:indefinite_power_iteration_metric}), is also a good metric candidate for the case when $\mat A$ is symmetric indefinite since $( {\lambda}_{x}^T {\lambda}_{x})^{1/2}$ is also positive definite in this case.

The proposed metric in (\ref{eq:indefinite_power_iteration_metric}) with $\omega = 0$ generalizes the well-known \emph{power iteration algorithm} for computing the dominant eigenvalues of a matrix \citep[Section~8.2.1]{golub96a}. To show this consider the case $\mat{B} = \mat{I}$. Given an iterate $x = \mat{X}$ such that $\mat{X}^T \mat{X} = \mat{I}$, the update of the Riemannian steepest-descent algorithm \changeBMM{with} unit step-­size has the characterization (after a few computations) that is equivalent to $x_+ = {\rm qf}(\mat{X}(\mat{I}  + {\lambda}_{x}( {\lambda}_{x}^T {\lambda}_{x})^{-1/2})  - \mat{AX} ( {\lambda}_{x}^T {\lambda}_{x})^{-1/2} )$. Locally, in the neighborhood of a minimum, $\mat{I}  + {\lambda}_{x}( {\lambda}_{x}^T {\lambda}_{x})^{-1/2} \approx 0$, and therefore, the equivalent update is \changeBMM{$x_+ = {\rm qf}(\mat{AX})$} which is the classical power iteration update \citep[Section~8.2.1]{golub96a}. In other words, the power algorithm has the interpretation of a Riemannian steepest-descent algorithm with the metric from (\ref{eq:indefinite_power_iteration_metric}) and with $\omega = 0$. Similarly, the steepest-descent algorithm with a shifted version of the metric (\ref{eq:indefinite_power_iteration_metric}), i.e., for $\omega$ updated with iterations, generalizes the GRQI algorithm proposed by \citet{absil02a}.

A similar insight still holds when the quadratic cost is generalized to a strictly concave function, i.e., minimizing a concave cost (or maximizing a convex cost) with orthogonality constraints. For the metric with $\omega = 0$, i.e., taking only the constraint-related term, this is the essence of the \emph{generalized power method} proposed by \citet{journee10b}.

\subsection{A numerical illustration}\label{sec:grassmann}

\begin{figure*}[t]
\center
\includegraphics[scale = 0.28]{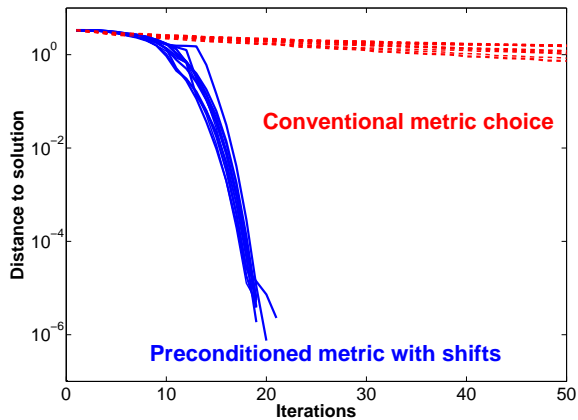}
\caption{Benefits of the proposed metric (\ref{eq:inverse_iteration_metric}) for the generalized eigenvalue problem to compute the extreme $5$-dimensional subspace (corresponding to the smallest $5$ eigenvalues) of the matrix pencil $(\mat{A}, \mat{B})$ of size $500 \times 500$. The problem instance is described in Section \ref{sec:grassmann}. Shown are $10$ runs of the Riemannian steepest-descent algorithms with random initializations for the problem instance. The distance to the solution is defined as the square root of the sum of \change{squared} canonical angles between the current subspace and the dominant $5$-dimensional subspace of $\mat{B}^{-1}\mat{A}$.}
\label{fig:grassmann}
\end{figure*}

As a numerical comparison, we consider the example proposed by \citet[Section~8]{manton02a}. $\mat{A}$ is a diagonal matrix of size $500 \times 500$ with entries \changeBM{equispaced} on the interval $[10, 11]$. $\mat{B}$ is chosen as the identity matrix of size $500 \times 500$. In Figure \ref{fig:grassmann}, we seek to compute the $r=5$ smallest eigenvalues of $\mat{B}^{-1}\mat{A}$. The algorithms compared are the Riemannian steepest-descent algorithms with the Euclidean metric (\ref{eq:grassmann_metric_standard}) and the preconditioned Riemannian metric in (\ref{eq:inverse_iteration_metric}) with the $\omega$-updating procedure (\ref{eq:update_omega}). Both the algorithms are stopped when either the norm of the gradient is below $10^{-8}$ or when they complete $500$ iterations. Distances of the iterates to the solution \change{are} plotted for the algorithms. The distance of an iterate $\mat{X}$ to the solution $\mat{X}_{\rm opt}$ is defined as the square root of the sum of canonical angles between $\mat{X}$ and $\mat{X}_{\rm opt}$. In Matlab it is computed using the command \texttt{norm(acos(svd(orth(X)'*orth(Xopt))))}. Figure \ref{fig:grassmann} \changeBMM{shows the initial $50$ iterations, where we see} that tuning the metric to the problem structure leads to improved performance.


\section{Quadratic optimization with rank constraints}\label{sec:quadratic_lowrank}

This class of problems has met with considerable interest in recent years. Applications include collaborative filtering \citep{rennie05a}, multivariate linear \changeBMM{classification} \citep{amit07a}, dimensionality reduction \citep{cai07a}, learning of low-rank distances \citep{kulis09a, meyer11c, mishra11b}, filter design problems \citep{manton02a}, model reduction in dynamical systems \citep{benner13a, li04a, vandereycken10a}, sparse principal components analysis \citep{burer03a, journee10a}, computing maximal cut of a graph \citep{burer03a, journee10a}, and low-rank matrix completion \citep{keshavan10a, ngo12a, boumal11a, mishra14a, vandereycken13a}, to name just a few. 

In all those applications, the discussion in Section \ref{sec:connection} allows us to propose novel preconditioned Riemannian metrics. These metrics connect to those proposed by \citet{mishra12a, mishra14d,ngo12a, mishra14c} for specific optimization problems.




A popular way to characterize the set of fixed-rank matrices is through fixed-rank matrix factorizations. Most matrix factorizations have symmetry properties that make them \changeBMM{non-unique}. And in many cases the set of rank $r$ of $n\times m$ matrices $\mathbb{R}_r^{n \times m}$ is identified with structured (smooth and differentiable) quotient spaces \citep{mishra14a, meyer11b,absil14a}. Figure \ref{fig:factorizations} shows three different fixed-rank matrix factorizations and the quotient manifold structure of the set $\mathbb{R}_r^{n \times m}$.


To identify Riemannian metrics on the low-rank manifold $\mathbb{R}_r^{n \times m}$, we consider minimization of a convex quadratic cost function. Specifically, we focus on the parameterization $\mat{X} = \mat{GH}^T$, where $\mat{X} \in \mathbb{R}_r ^{n \times m}$, $\mat{G} \in \mathbb{R}_*^{n \times r}$ (the set of full column rank matrices), and $\mat{H} \in \mathbb{R}_*^{m \times r}$. Other fixed-rank matrix factorizations are dealt with similarly.

Consider the optimization problem
\begin{equation}\label{eq:lowrank_optimization_problem}
\begin{array}{lll}
\min\limits_{\mat{X} \in \mathbb{R}_r ^{n \times m}} &  \frac{1}{2} \trace(\mat{X}^T \mat{A} \mat{X} \mat{B}) + \trace(\mat{X}^T \mat{C}), \\
\end{array}
\end{equation}
where $\mat{A} \succ 0 $ of size $n\times n$, $\mat{B} \succ 0 $ of size $m\times m$, and  $\mat{C} \in \mathbb{R}^{n \times m}$. Positive definiteness of $\mat{A}$ and $\mat{B}$ implies that the cost function is bounded from below and is \emph{convex} in $\mat{X}$. Invoking the low-rank parameterization $\mat{X} = \mat{GH}^T$, shown in Figure \ref{fig:factorizations}, the problem (\ref{eq:lowrank_optimization_problem}) translates to 
\begin{equation*}\label{eq:lowrank_optimization_problem_quotient}
\begin{array}{lll}
\min\limits_{(\mat{G}, \mat{H}) \in \mathbb{R}^{n \times r} \times \mathbb{R}^{m \times r}} &  \frac{1}{2} \trace(\mat{H} \mat{G}^T\mat{AGH}^T\mat{B} ) + \trace (\mat{H} \mat{G}^T \mat{C}) \\
\subject & [(\mat{G}, \mat{H})] \in \mathbb{R}_*^{n \times r} \times \mathbb{R}_*^{m \times r} / \GL{r},
\end{array}
\end{equation*}
where the equivalence class $[(\mat{G}, \mat{H})] := \{ (\mat{GM}^{-1}, \mat{HM}^T):  \mat M \in \GL{r} \}$ and $\GL{r}$ is the set of $r\times r$ square matrices of non-zero determinant.

\begin{figure}[t]
	\centering
	\includegraphics[scale = 0.60]{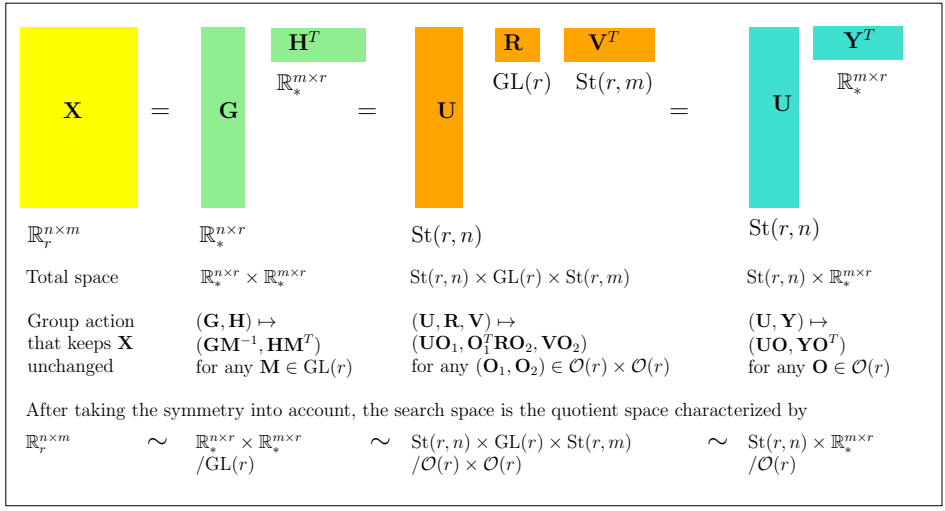}
	\caption{Fixed-rank matrix factorizations lead to quotient search spaces due to their intrinsic symmetries. The pictures emphasize the situation of interest, i.e., the rank $r$ is small compared to the matrix dimensions. $\Stiefel{r}{n}$ is the set of $n\times r$ matrices with orthogonal columns, $\mathbb{R}_*^{n\times r}$ is the set of $n\times r$ matrices with full column rank, $\GL{r}$ is the set of $r\times r$ square matrices with non-zero determinant, and $\OG{r}$ is the set of $r\times r$ square matrices with \changeBMM{orthonormal} columns and rows.}
	\label{fig:factorizations}
\end{figure}

A conventional way to handle this symmetry in the Riemannian framework is  endowing the set $\mathbb{R}_*^{n\times r}$ with the \emph{natural metric} \citep[Section~3.6.4]{absil08a}. Since the computational space $ {\mathcal{M}}$ is the product space $\mathbb{R}_*^{n \times r} \times \mathbb{R}_*^{m \times r}$, the natural metric is 
\begin{equation*}\label{eq:lowrank_metric_standard}
\begin{array}{lll}
g_x(\eta_x, \xi_x) = \trace((\mat{G}^T\mat{G})^{-1} \eta_{\mat G}^T \xi_{\mat G}) + \trace((\mat{H}^T\mat{H})^{-1} \eta_{\mat H}^T \xi_{\mat H}),
\end{array}
\end{equation*}
where $x$ has the matrix representation $(\mat{G}, \mat{H}) \in \mathbb{R}_*^{n \times r} \times \mathbb{R}_*^{m\times r}$ and $\xi_x, \eta_x$ are vectors belonging to the tangent space $\mathbb{R}^{n\times r}\times \mathbb{R}^{m \times r}$, i.e., $\xi_x$ has the matrix representation $(\xi_{\mat G}, \xi_{\mat H})\in \mathbb{R}^{n\times r} \times 
\mathbb{R}^{m \times r}$.

In contrast, we follow the developments in Section \ref{sec:convexity} to propose a family of metrics that takes the problem structure into account by exploiting the structure of the regularized Lagrangian (\ref{eq:regularized_Lagrangian}). Since the set $\mathbb{R}_*^{n \times r} \times \mathbb{R}_*^{m\times r}$ is an open subset of the space $\mathbb{R}^{n \times r} \times \mathbb{R}^{m \times r}$, the Lagrangian only consists of the cost function, i.e.,
\begin{equation}\label{eq:lagrangian_lowrank}
\begin{array}{lrll}
&  {\mathcal L}( {x}) & = & \trace(\mat{H} \mat{G}^T\mat{AGH}^T\mat{B} )/ 2 + \trace (\mat{H} \mat{G}^T \mat{C}) \\
\Rightarrow &  {\mathcal L}_x( {x})  & =
& (\mat{AGH}^T\mat{BH} + \mat{CH} , \mat{BHG}^T\mat{AG} + \mat{C}^T\mat{G})\\
\Rightarrow &  \D ^2 {\mathcal L}( {x}) [\xi_x] &=
& (\mat{A}\xi_{\mat G} \mat{H}^T\mat{BH} + 2\mat{AG}\Sym({\mat H}^T \mat{B} \xi_{\mat H}) + \mat{C}\xi_{\mat H}, \\
& & & \mat{B}\xi_{\mat H} \mat{G}^T\mat{AG} + 2\mat{BH}\Sym({\mat G}^T \mat{A} \xi_{\mat G}) + \mat{C}^T\xi_{\mat G} ),
\end{array}
\end{equation}
where $x$ has the matrix representation $(\mat{G}, \mat{H}) \in \mathbb{R}_*^{n \times r} \times \mathbb{R}_*^{m\times r}$, $\xi_x$ has the matrix representation $(\xi_{\mat G}, \xi_{\mat H})\in \mathbb{R}^{n\times r} \times 
\mathbb{R}^{m \times r}$, $\mathcal{L}_x (x)$ is the first-order derivative of $\mathcal{L}(x)$, $\D ^2 \mathcal{L}(x)[\xi_x]$ is the second-order derivative of $\mathcal{L}(x)$ applied in the direction $\xi_x$, and $\Sym(\cdot)$ extracts the symmetric part of a square matrix, i.e., $\Sym(\mat D) = (\mat{D}^T + \mat{D})/2$. 

It is readily checked that the Lagrangian $\mathcal{L}(x)$ in (\ref{eq:lagrangian_lowrank}) remains unchanged under the transformation $(\mat{G},\mat{H}) \mapsto (\mat{GM}^{-1}, \mat{HM}^T)$ for all $\mat{M} \in \GL{r}$. Subsequently, we have the following proposition for constructing a family of Riemannian metrics for (\ref{eq:lowrank_optimization_problem}) on the fixed-rank quotient manifold.

\begin{proposition}\label{prop:lowrank}
\changeBMM{There exists a family of Riemannian metrics}
\begin{equation}\label{eq:lowrank_metric_Hessian}
\begin{array}{rll}
 {g}_{  x} ( {\xi}_{  x},  {\eta}_{  x} )= &   \omega_1\langle   {\eta}_{\mat G},   \mat{A}\xi_{\mat G} \mat{H}^T\mat{BH} \rangle\\
 & +  \omega_2\langle     {\eta}_{\mat G}, 2\mat{AG}\Sym({\mat H}^T \mat{B} \xi_{\mat H}) + \mat{C}\xi_{\mat H} \rangle  \\
  & + \omega_3\langle  {\eta}_{\mat H},\mat{B}\xi_{\mat H} \mat{G}^T\mat{AG} \rangle\\
   & + \omega_4 \langle  {\eta}_{\mat H}, 2\mat{BH}\Sym({\mat G}^T \mat{A} \xi_{\mat G}) + \mat{C}^T\xi_{\mat G} \rangle, \\ 

\end{array}
\end{equation}
\changeBMM{on $\mathbb{R}_*^{n \times r} \times \mathbb{R}_*^{m \times r}$ with $\omega_1, \omega_2, \omega_3,\omega_4 \in [0, 1]$, each of which induces a Riemannian metric on the quotient manifold $\mathbb{R}_*^{n \times r} \times \mathbb{R}_*^{m \times r} / \GL{r}$ in the neighborhood of the local minimum of (\ref{eq:lowrank_optimization_problem}).} Here ${x} = (\mat{G}, \mat{H}) \in  \mathcal{M}$,  $\mathcal{M} = \mathbb{R}_*^{n \times r} \times \mathbb{R}_*^{m \times r} $, and ${\xi}_{  x}, \eta_x$ are vectors in the tangent space $T_x \mathcal{M}$.
\end{proposition}

\begin{proof}
First, we show that the metrics from (\ref{eq:lowrank_metric_Hessian}) respect the condition in (\ref{eq:metric_quotient}). In order to show that the metric does not change along the equivalence class $[x] = [(\mat{G}, \mat{H})] = \{ (\mat{GM}^{-1}, \mat{HM}^T):  \mat M \in \GL{r} \}$ for all $\mat{M} \in \GL{r} $, it is equivalent, but simplified following \citep[Proposition~3.6.1]{absil08a}, to show that the metric for tangent vectors ${\xi}_{  x}, \eta_x \in T_x \mathcal{M}$ does not change under the transformations $(\mat{G}, \mat{H}) \mapsto (\mat{GM}^{-1}, \mat{HM}^T)$, $(\eta_{\mat G}, \eta_{\mat H}) \mapsto (\eta_{\mat G} \mat{M}^{-1}, \eta_{\mat H} \mat{M}^T)$, and $(\xi_{\mat G}, \xi_{\mat H}) \mapsto (\xi_{\mat G} \mat{M}^{-1}, \xi_{\mat H} \mat{M}^T)$. A few extra computations show that indeed the metrics from (\ref{eq:lowrank_metric_Hessian}) respect the condition in (\ref{eq:metric_quotient}).

Second, we show the construction of one particular family of Riemannian metrics. To this end, consider the case in (\ref{eq:lowrank_metric_Hessian}) where $\omega_1 =1$ and $\omega_2 = \omega_3 =\omega_4 = \omega$. Restricting $\omega \in [0, 1)$ guarantees that in the neighborhood of a minimum $g_x(\zeta_x , \zeta_x) > 0$ for all $\zeta_x \in T_x \mathcal{M}$, satisfying the criterion of positive definiteness on the tangent space. 

The proofs for symmetry compatibility and positive definiteness on the tangent space \changeBMM{conclude} the proof of the proposition.
\end{proof}

%

\begin{table}[H]
\begin{center} \small
\begin{tabular}{ p{5cm} | p{9cm} }
&  
$
\begin{array}[t]{lll}
\min\limits_{\substack{\mat{G} \in \mathbb{R}_*^{n \times r} \\ \mat{H} \in \mathbb{R}_*^{m \times r}}}   \trace(\mat{H} \mat{G}^T\mat{AGH}^T\mat{B} )/2+ \trace (\mat{H} \mat{G}^T \mat{C}) \\
\end{array}
$
\\
\hline
& \\
$\begin{array}[t]{lll} {\rm Matrix\ representation}\\
{\rm of\ an\ element\ } x \in \mathcal{M} \end{array}$& $\begin{array}[t]{lll} {x} = (\mat{G}, \mat{H})\end{array}$\\
& \\
$\begin{array}[t]{lll} {\rm  Computational\  space\  }  {\mathcal M} \end{array}$&
$
\begin{array}[t]{ll}
\mathbb{R}_*^{n \times r} \times \mathbb{R}_*^{m\times r}
\end{array}
$
\\
& \\
$\begin{array}[t]{lll} {\rm Group\ action }\end{array}$ &
$\begin{array}[t]{lll}(\mat{GM}^{-1}, \mat{HM}^T) \end{array},\ \forall\mat{M} \in  \GL{r}$ 
\\
& \\
$\begin{array}[t]{lll} {\rm Quotient\ space\ }  \end{array}$ &  
$
\begin{array}[t]{lll}
\mathbb{R}_*^{n \times r} \times \mathbb{R}_*^{m\times r}/  \GL{r} 
\end{array}
$
\\
& \\

$
\begin{array}[t]{ll}
{\rm Tangent\ vectors\ in\ } 
T_{  x} {\mathcal M}
\end{array}
$ & $
\begin{array}[t]{lll}
 {\xi}_{  x}  = (\xi_{\mat G}, \xi_{\mat H}) \in  
\mathbb{R}^{n \times r} \times \mathbb{R}^{m\times r} 
\end{array}
$
 \\
 & \\
$
\begin{array}[t]{lll}
{\rm Metric \ } {g}_{  x} ( {\xi}_{  x},   {\zeta}_{  x}) \\
{\rm for  \ }  {\xi}_{  x},  {\zeta}_{  x} \in T_{  x}  {\mathcal M}
\end{array}
$
& 
$
\begin{array}[t]{lll}
 {g}_{  x} ( {\xi}_{  x},  {\eta}_{  x} )=  \omega_1 \langle   {\eta}_{\mat G},   \mat{A}\xi_{\mat G} \mat{H}^T\mat{BH}\rangle \\
 \quad \quad \quad \quad \quad  + \omega_2\langle {\eta}_{\mat G},   2\mat{AG}\Sym({\mat H}^T \mat{B} \xi_{\mat H}) + \mat{C}\xi_{\mat H} \rangle  \\
 \quad \quad \quad \quad \quad + \omega_3\langle  {\eta}_{\mat H},\mat{B}\xi_{\mat H} \mat{G}^T\mat{AG} \rangle\\
 \quad \quad \quad \quad \quad + \omega_4 \langle \eta_{\mat H}, 2\mat{BH}\Sym({\mat G}^T \mat{A} \xi_{\mat G}) + \mat{C}^T\xi_{\mat G} \rangle, \\ 
 \\
{\rm or\ the\ metrics\ proposed\ in\ Section\ } \ref{sec:metric_tuning_lowrank}
 
\end{array}
$\\
& \\
$
\begin{array}[t]{lll}
{\rm Cost\ function }\\
\end{array}
$

&
$
\begin{array}[t]{ll}
 {f}(  x) = \trace(\mat{H} \mat{G}^T\mat{AGH}^T\mat{B} )/2+ \trace (\mat{H} \mat{G}^T \mat{C})
\end{array}
$
\\

& \\
$
\begin{array}[t]{lll}
{\rm First}\mbox{-}{\rm order\ derivative\ of }\\
 {f}(  x)
\end{array}
$

&
$
\begin{array}[t]{ll}
 {f}_x(  x) = (\mat{SH}, \mat{S}^T \mat{G}),\\
{\rm where\ } \mat{S} = \mat{A}\mat{GH}^T \mat{B} + \mat{C}
\end{array}
$
\\
& \\
$
\begin{array}[t]{lll}
{\rm Search\ direction\ }
\end{array}
$

&
$
\begin{array}[t]{ll}
\argmin\limits_{ {\zeta}_{  x} \in T_{  x} {\mathcal M}} \quad {f}(  x) + \langle  {f}_x(  x),  {\zeta}_{  x}\rangle +\frac{1}{2}  {g}_{  x}( {\zeta}_{  x},  {\zeta}_{  x})\\
\end{array}
$
\\
& \\

$
\begin{array}[t]{lll}
{\rm Retraction\ }  {R}_{  x}( {\xi}_{  x}){\rm \ that } \\

{\rm maps\ a\ search\ direction\ } {  \xi}_{  x}\\
{\rm onto\ }  {\mathcal M}
\end{array}
$
&
$
\begin{array}[t]{lll}
(\mat{G} +  {\xi}_{ \mat G}, \mat{H} + \xi_{\mat H})
\end{array}
$
\\
 & \\
 \hline
\end{tabular}
\end{center} 
\caption{Optimization-related ingredients for the problem (\ref{eq:lowrank_optimization_problem}). The numerical complexity per iteration of the Riemannian steepest-descent algorithm depends on solving for $\zeta_x$ for the search direction computation. For example, sparsity in matrices $\mat{A}$ and $\mat{B}$ considerably reduces the computation cost. The retraction mapping is the \changeBMM{Cartesian} product of the standard
retraction mapping on the manifold $\mathbb{R}_*^{n \times r}$ \citep[Example~3.6.4]{absil08a}. Few choices of the regularizing parameters $\omega_1, \omega_2, \omega_3, \omega_4 \in [0, 1]$ for relevant situations are discussed in Section \ref{sec:metric_tuning_lowrank}.}
\label{tab:lowrank} 
\end{table}

Matrix characterizations of various optimization-related ingredients are summarized in Table \ref{tab:lowrank}. The retraction operator is the standard generalization of the retraction operator on the manifold $\mathbb{R}_*^{n\times r}$ defined by \citet[Example~3.6.4]{absil08a}. 

It should be noted that numerical performance of algorithms \changeBMM{depends} on computing the Riemannian gradient efficiently with the metric (\ref{eq:lowrank_metric_Hessian}). This may become a numerically cumbersome task due to a number of coupled terms that are involved in (\ref{eq:lowrank_metric_Hessian}). However, below we show that the problem structure can be further exploited to decompose the metric (\ref{eq:lowrank_metric_Hessian}) into a locally dominating part with a simpler metric structure and a \emph{weighted} remainder. The dominant approximation may be preferred in a number of situations.

\subsection{Metric tuning and shift policies} \label{sec:metric_tuning_lowrank}
It should be emphasized that the cost function in (\ref{eq:lowrank_optimization_problem}) is \emph{convex and quadratic} in $\mat X$. Consequently, the cost function is also convex and quadratic in the arguments $(\mat{G}, \mat{H})$ \emph{individually}. As a consequence, the block diagonal elements of the second-order derivative $\mathcal{L}_{xx}(x)$ of the Lagrangian (\ref{eq:lagrangian_lowrank}) are strictly positive definite. This enables us to construct a family of Riemannian metrics with \emph{shifts} of the form
\begin{equation}\label{eq:block_metric_lowrank}
\begin{array}{lll}
 {g}_{  x} ( {\xi}_{  x},  {\eta}_{  x} )= & \omega \langle   {\eta}_{\mat G}, 2\mat{AG}\Sym({\mat H}^T \mat{B} \xi_{\mat H}) + \mat{C}\xi_{\mat H} \rangle  \\
 
&+ \omega \langle  {\eta}_{\mat H},2\mat{BH}\Sym({\mat G}^T \mat{A} \xi_{\mat G}) + \mat{C}^T\xi_{\mat G} \rangle  \\ 
& + \underbrace{\langle   {\eta}_{\mat G},   \mat{A}\xi_{\mat G} \mat{H}^T\mat{BH}\rangle  + \langle  {\eta}_{\mat H},\mat{B}\xi_{\mat H} \mat{G}^T\mat{AG}\rangle}_{{\rm from\ block\ diagonal\ approximation\ of\ }  \mathcal{L}_{xx}(x) }  \\ 
\end{array}
\end{equation}
in the neighborhood of a minimum, where ${x} = (\mat{G}, \mat{H}) \in  \mathbb{R}_*^{n \times r} \times \mathbb{R}_*^{m \times r} $, ${\xi}_{  x}, \eta_x$ are tangent vectors in $\mathbb{R}^{n\times r}\times \mathbb{R}^{m \times r}$, and $\omega \in [0, 1)$. \change{The form in (\ref{eq:block_metric_lowrank}) is derived from (\ref{eq:lowrank_metric_Hessian}) by choosing $\omega_1 = \omega_3 = 1$ and $\omega_2 = \omega_4 = \omega$. It should be noted that the case $\omega = 1$ is discarded to prevent singularity of the metric candidate. }

Away from the neighborhood, the metric (\ref{eq:block_metric_lowrank}) with $\omega = 0$ becomes a good metric candidate as $\mat{H}^T\mat{BH}$ and $\mat{G}^T \mat{AG}$ are positive definite for all $(\mat{G}, \mat{H}) \in \mathbb{R}_*^{n \times r}\times \mathbb{R}_*^{m \times r}$. The other benefit of $\omega$ being $0$ is that the resulting metric has a \emph{simpler matrix characterization}, and hence it may be preferred in numerically demanding instances.

\subsection{Symmetric positive definite matrices}
A popular subset of fixed-rank matrices is the set of symmetric positive semidefinite matrices \citep{burer03a, journee10a, meyer11c, vandereycken10a}. The set $\PSD{r}{n}$, the set of rank-$r$ symmetric positive semidefinite matrices of size $n\times n$, is equivalent to the set $\mathbb{R}_r^{n\times m}$ with symmetry imposed on the rows and columns, and therefore, it admits a number of factorizations similar to those in Figure \ref{fig:factorizations}. Consequently, the low-rank parameterization discussed earlier, in the context of the general case, has the counterpart $\mat{X} = \mat{YY}^T$, where $\mat{X} \in \PSD{r}{n}$ and $\mat{Y} \in \mathbb{R}_*^{n\times r}$ (full column rank matrices of size $n\times r$). This parameterization is not unique as $\mat{X} \in \PSD{r}{n} = \mat{YY}^T$ remains unchanged under the transformation $\mat{Y} \mapsto \mat{YO}$ for all $O \in \OG{r}$, where $\OG{r}$ is set of orthogonal matrices of size $r\times r$ such that $\mat{OO}^T = \mat{O}^T\mat{O} = \mat{I}$. The resulting search space is, thus, the set of equivalence classes $[\mat{Y}] = \{\mat{YO}: \mat{O} \in \OG{r} \}$ and is the quotient manifold $\mathbb{R}_*^{n \times r}/\OG{r}$ \citep{journee10a}. 

The following proposition summarizes the discussion on Riemannian metrics for the case of symmetric positive semidefinite matrices.

\begin{proposition}\label{prop:psd}
Consider the optimization problem
\begin{equation}\label{eq:psd_optimization_problem}
\begin{array}{lll}
\min\limits_{\mat{X} \in \mathbb{R}^{n\times n}} &  \frac{1}{2} \trace(\mat{X} \mat{A} \mat{X} \mat{B}) + \trace(\mat{X} \mat{C}) \\
\subject & \mat{X} \in \PSD{r}{n},
\end{array}
\end{equation}
where $\mat{A}, \mat{B} \succ 0 $ of size $n\times n$ and  $\mat{C} \in \mathbb{R}^{n \times n}$ is a symmetric matrix. Consider also the factorization  $\mat{X} = \mat{YY}^T$ of rank-$r$ symmetric positive semidefinite matrices to encode the rank constraint, where $\mat{Y} \in \mathbb{R}_*^{n \times r}$ (full column rank matrices).

\changeBMM{There exists a family of Riemannian metrics}
\begin{equation}\label{eq:block_metric_psd}
\begin{array}{lll}
 {g}_{  x} ( {\xi}_{  x},  {\eta}_{  x} )= &  \omega \langle   {\eta}_x,    2\mat{AY}\Sym({\mat Y}^T \mat{B} \xi_{x}) + 2\mat{BY}\Sym({\mat Y}^T \mat{A} \xi_{x}) + 2\mat{C}\xi_{x} \rangle \\
  & +  \underbrace{\langle \eta_x, \mat{A}\xi_{x} \mat{Y}^T\mat{BY}  + \mat{B}\xi_{x} \mat{Y}^T\mat{AY}   \rangle,}_{{\rm  Dominant\ positive\ definite\ approximation\ of\ } \mathcal{L}_{xx}(x) } \\ 

\end{array}
\end{equation}
\changeBMM{on $\mathbb{R}_*^{n \times r}$ with $\omega \in [0, 1)$, each of which induces a Riemannian metric on $\mathbb{R}_*^{n \times r}/\OG{r}$ in the neighborhood of the minimum of the problem (\ref{eq:psd_optimization_problem}).} Here ${x} = \mat{Y} \in  \mathbb{R}_*^{n \times r} $, ${\xi}_{  x}, \eta_x$ are vectors in the tangent space $\mathbb{R}^{n\times r}$, and $\mathcal{L}_{xx}(x)$ is the second-order derivative of the Lagrangian. Beyond the neighborhood, the metric (\ref{eq:block_metric_psd}) with $\omega = 0$ becomes a good metric candidate as $\mat{Y}^T\mat{BY}$ and $\mat{Y}^T \mat{AY}$ are positive definite for all $\mat{Y} \in \mathbb{R}_*^{n \times r}$.
\end{proposition}

\begin{proof}
The proof follows from the discussion in Section \ref{sec:metric_tuning_lowrank}.
\end{proof}

\subsection{A numerical illustration}
We showcase the Riemannian preconditioning approach for computing low-rank solutions to the generalized Lyapunov equation of the form
\begin{equation}\label{eq:lowrank_Lyapunov}
\mat{AXB} + \mat{BXA} = \mat{C},
\end{equation}
where $\mat{A}, \mat{B} \succ 0$, and $\mat C$ is a low-rank symmetric positive semidefinite matrix. Matrices have appropriate dimensions. $\mat A$ is referred to as the \emph{system matrix} and $\mat B$ is referred to as the \emph{mass matrix}. The solution to (\ref{eq:lowrank_Lyapunov}) is expected to be low rank and symmetric positive semidefinite \citep{benner13a, li04a, vandereycken10a}. 

To compute low-rank solutions to (\ref{eq:lowrank_Lyapunov}), we minimize the \emph{energy norm} \change{$\trace(\mat{XAXB})\allowbreak - \trace(\mat{XC})$} over $\PSD{r}{n}$ \citep{vandereycken10a}. Proposition \ref{prop:psd} allows to characterize a family of metrics in (\ref{eq:block_metric_psd}) for solving the generalized Lyapunov equation. In contrast, an alternative is to consider the Euclidean metric, i.e.,
\begin{equation}\label{eq:lowrank_Lyapunov_metric_standard}
\begin{array}{lll}
 {g}_{  x} ( {\xi}_{  x},  {\zeta}_{  x} )=  \trace(  {\zeta}_{  x}^T   {\xi}_{  x} ) ,  \\ 
\end{array}
\end{equation}
where $ {x} = \mat{Y}$ and $ {\xi}_{  x}$ and $ {\zeta}_{  x}$ are tangent vectors. This is, for example, the Riemannian metric proposed by \citet{journee10a}. It is invariant to the group action $\mat{Y} \mapsto \mat{YO}$ for all $\mat{O} \in \OG{r}$. Although the alternative choice (\ref{eq:lowrank_Lyapunov_metric_standard}) is appealing for its numerical simplicity, the following test case clearly illustrates the benefits of the Riemannian preconditioning approach. 

\begin{figure*}[t]
\center
\includegraphics[scale = 0.28]{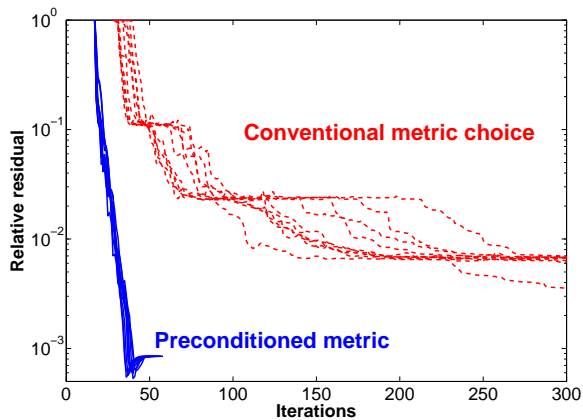}
\caption{The generalized low-rank Lyapunov equation problem (\ref{eq:lowrank_Lyapunov}). The test case is the benchmark problem from \citep[Example~2.1]{penzl99a} with $n=500$. The proposed Riemannian preconditioning approach with the metric (\ref{eq:block_metric_psd}) and $\omega = 0$ drastically improves the performance over the algorithm based on the Euclidean metric (\ref{eq:lowrank_Lyapunov_metric_standard}). Additionally, the choice $\omega =0$ leads to a simpler metric structure that can be exploited in a large-scale setup. \change{Here we show the convergence of the relative residual $\|\mat{AXB} + \mat{BXA} - \mat{C} \|_{\Fro} / \|{\mat C} \|_{\Fro}$ (different from the cost function $\trace(\mat{XAXB}) - \trace(\mat{XC})$) that is often used as a measure of recovery.}}
\label{fig:lowrank_lyapunov}
\end{figure*}

We consider the standard benchmark problem from \citep[Example~2.1]{penzl99a} that corresponds to discretization of a one-dimensional heat equation from heat flow in a thin rod. For this example, $\mat{A}$ is a tridiagonal matrix of size $500\times 500$. The main diagonal of $\mat A$ has all the elements equal to $2$. In addition, the first diagonals below and above the main diagonal of $\mat A$ have all the entries equal to $-1$. $\mat A$ is an ill-conditioned matrix with condition number $10^5$. The mass matrix $\mat{B}$ is an identity matrix of size $500\times 500$. The matrix $\mat C$ is a rank one matrix of the form $ee^T$, where $e^T$ is a row vector of length $500$ of the form $[0\ 0\ \ldots \ 0\ 1]$. We seek to find a rank-$5$ \change{matrix} that best solves the generalized Lyapunov equation (\ref{eq:lowrank_Lyapunov}). Both the algorithms are stopped when either the norm of the gradient is below $10^{-8}$ or when they complete $500$ iterations. The plots in Figure \ref{fig:lowrank_lyapunov} show the progress of relative residual $\|\mat{AXB} + \mat{BXA} - \mat{C} \|_{\Fro} / \|{\mat C} \|_{\Fro}$ with iterations over $10$ random initializations, where $\mat{X} = \mat{YY}^T$. The Riemannian algorithm with the metric (\ref{eq:block_metric_psd}) and $\omega = 0$ convincingly outperforms the algorithm based on the Euclidean metric (\ref{eq:lowrank_Lyapunov_metric_standard}) in Figure \ref{fig:lowrank_lyapunov} for a number of runs, \changeBMM{where the initial $300$ iterations are shown for clarity.}

\section{Conclusion}\label{sec:conclusion}
This paper addresses the important issue of selecting a metric in the Riemannian optimization framework \changeBMM{on a quotient manifold}. We have shown that sequential quadratic programming provides an insight into selecting a family of Riemannian metrics that \change{takes into account} second-order information of the problem. Quadratic optimization with orthogonality or rank constraints provides a class of nonconvex problems for which the method is particularly insightful, thanks to local convexity of the cost and constraint when taken separately. In those instances, Riemannian preconditioning connects to a number of existing algorithms and provides a geometric interpretation of a number of ``shift'' policies in numerical linear algebra.

\section*{Acknowledgments}
We thank the editor and two anonymous reviewers for carefully checking the paper and providing a number of helpful remarks.

\bibliographystyle{spbasic}
\bibliography{arXiv_MS16_sqp_manifold}

\begin{thebibliography}{39}
\providecommand{\natexlab}[1]{#1}
\providecommand{\url}[1]{{#1}}
\providecommand{\urlprefix}{URL }
\expandafter\ifx\csname urlstyle\endcsname\relax
  \providecommand{\doi}[1]{DOI~\discretionary{}{}{}#1}\else
  \providecommand{\doi}{DOI~\discretionary{}{}{}\begingroup
  \urlstyle{rm}\Url}\fi
\providecommand{\eprint}[2][]{\url{#2}}

\bibitem[{Absil and Van~Dooren(2010)}]{absil10a}
Absil PA, Van~Dooren P (2010) Two-sided {Grassmann-Rayleigh} quotient
  iteration. Numerische Mathematik 114(4):549--571

\bibitem[{Absil et~al.(2002)Absil, Mahony, Sepulchre, and
  Van~Dooren}]{absil02a}
Absil PA, Mahony R, Sepulchre R, Van~Dooren P (2002) A {Grassmann-Rayleigh}
  quotient iteration for computing invariant subspaces. SIAM Review
  44(1):57--73

\bibitem[{Absil et~al.(2004)Absil, Sepulchre, Van~Dooren, and
  Mahony}]{absil04c}
Absil PA, Sepulchre R, Van~Dooren P, Mahony R (2004) Cubically convergent
  iterations for invariant subspace computation. SIAM Journal on Matrix
  Analysis and Applications 26(1):70--96

\bibitem[{Absil et~al.(2008)Absil, Mahony, and Sepulchre}]{absil08a}
Absil PA, Mahony R, Sepulchre R (2008) Optimization Algorithms on Matrix
  Manifolds. Princeton University Press, Princeton, NJ

\bibitem[{Absil et~al.(2009)Absil, Trumpf, Mahony, and Andrews}]{absil09b}
Absil PA, Trumpf J, Mahony R, Andrews B (2009) All roads lead to {N}ewton:
  Feasible second-order methods for equality-constrained optimization. Tech.
  rep., UCL-INMA-2009.024

\bibitem[{Absil et~al.(2014)Absil, Amodei, and Meyer}]{absil14a}
Absil PA, Amodei L, Meyer G (2014) {Two Newton methods on the manifold of
  fixed-rank matrices endowed with Riemannian quotient geometries}.
  Computational Statistics 29(3--4):569--590

\bibitem[{Amit et~al.(2007)Amit, Fink, Srebro, and Ullman}]{amit07a}
Amit Y, Fink M, Srebro N, Ullman S (2007) Uncovering shared structures in
  multiclass classification. In: Proceedings of the 24th International
  Conference on Machine Learning, pp 17--24

\bibitem[{Benner and Saak(2013)}]{benner13a}
Benner P, Saak J (2013) Numerical solution of large and sparse continuous time
  algebraic matrix {R}iccati and {L}yapunov equations: a state of the art
  survey. GAMM-Mitteilungen 36(1):32--52

\bibitem[{Boumal and Absil(2011)}]{boumal11a}
Boumal N, Absil PA (2011) {RTRMC}: A {R}iemannian trust-region method for
  low-rank matrix completion. In: Advances in Neural Information Processing
  Systems 24 ({NIPS}), pp 406--414

\bibitem[{Boumal et~al.(2014)Boumal, Mishra, Absil, and Sepulchre}]{boumal14a}
Boumal N, Mishra B, Absil PA, Sepulchre R (2014) Manopt: a {M}atlab toolbox for
  optimization on manifolds. Journal of Machine Learning Research
  15(Apr):1455--1459

\bibitem[{Burer and Monteiro(2003)}]{burer03a}
Burer S, Monteiro R (2003) A nonlinear programming algorithm for solving
  semidefinite programs via low-rank factorization. Mathematical Programming
  95(2):329--357

\bibitem[{Cai et~al.(2007)Cai, He, and Han}]{cai07a}
Cai D, He X, Han J (2007) Efficient kernel discriminant analysis via spectral
  regression. In: IEEE International Conference on Data Mining (ICDM), pp
  427--432

\bibitem[{Edelman et~al.(1998)Edelman, Arias, and Smith}]{edelman98a}
Edelman A, Arias T, Smith S (1998) The geometry of algorithms with
  orthogonality constraints. SIAM Journal on Matrix Analysis and Applications
  20(2):303--353

\bibitem[{Eld\'en and Park(1999)}]{elden99a}
Eld\'en L, Park H (1999) A {P}rocrustes problem on the {Stiefel} manifold.
  Numerische Mathematik 82(4):599--619

\bibitem[{Gabay(1982)}]{gabay82a}
Gabay D (1982) Minimizing a differentiable function over a differential
  manifold. Journal of Optimization Theory and Applications 37(2):177--219

\bibitem[{Golub and Van~Loan(1996)}]{golub96a}
Golub GH, Van~Loan CF (1996) Matrix Computations, 3rd edn. The Johns Hopkins
  University Press, 2715 North Charles Street, Baltimore, Maryland 21218-4319

\bibitem[{Journ{\'e}e et~al.(2010)Journ{\'e}e, Bach, Absil, and
  Sepulchre}]{journee10a}
Journ{\'e}e M, Bach F, Absil PA, Sepulchre R (2010) Low-rank optimization on
  the cone of positive semidefinite matrices. SIAM Journal on Optimization
  20(5):2327--2351

\bibitem[{Journ\'ee et~al.(2010)Journ\'ee, Nesterov, Richt\'arik, and
  Sepulchre}]{journee10b}
Journ\'ee M, Nesterov Y, Richt\'arik P, Sepulchre R (2010) Generalized {P}ower
  method for sparse principal component analysis. Journal of Machine Learning
  Research 11(Feb):517--553

\bibitem[{Keshavan et~al.(2010)Keshavan, Montanari, and Oh}]{keshavan10a}
Keshavan RH, Montanari A, Oh S (2010) Matrix completion from a few entries.
  IEEE Transactions on Information Theory 56(6):2980--2998

\bibitem[{Kulis et~al.(2009)Kulis, Sustik, and Dhillon}]{kulis09a}
Kulis B, Sustik M, Dhillon IS (2009) Low-rank kernel learning with {B}regman
  matrix divergences. Journal of Machine Learning Research 10(Feb):341--376

\bibitem[{Li and White(2004)}]{li04a}
Li JR, White J (2004) Low-rank solution of {L}yapunov equations. SIAM Review
  46(4):693--713

\bibitem[{Manton(2002)}]{manton02a}
Manton J (2002) Optimization algorithms exploiting unitary constraints. IEEE
  Transactions on Signal Processing 50(3):635--650

\bibitem[{Meyer(2011)}]{meyer11b}
Meyer G (2011) Geometric optimization algorithms for linear regression on
  fixed-rank matrices. PhD thesis, University of Li\`ege, Li\`ege, Belgium

\bibitem[{Meyer et~al.(2011)Meyer, Bonnabel, and Sepulchre}]{meyer11c}
Meyer G, Bonnabel S, Sepulchre R (2011) Regression on fixed-rank positive
  semidefinite matrices: a {R}iemannian approach. Journal of Machine Learning
  Research 11(Feb):593--625

\bibitem[{Mishra and Sepulchre(2014)}]{mishra14c}
Mishra B, Sepulchre R (2014) {R3MC}: A {R}iemannian three-factor algorithm for
  low-rank matrix completion. In: Proceedings of the 53rd IEEE Conference on
  Decision and Control (CDC), pp 1137--1142

\bibitem[{Mishra and Vandereycken(2014)}]{mishra14d}
Mishra B, Vandereycken B (2014) A {R}iemannian approach to low-rank algebraic
  {R}iccati equations. In: Proceedings of the 21st International Symposium on
  Mathematical Theory of Networks and Systems (MTNS), pp 965--968

\bibitem[{Mishra et~al.(2011)Mishra, Meyer, and Sepulchre}]{mishra11b}
Mishra B, Meyer G, Sepulchre R (2011) Low-rank optimization for distance matrix
  completion. In: Proceedings of the 50th IEEE Conference on Decision and
  Control (CDC-ECC), pp 4455--4460

\bibitem[{Mishra et~al.(2012)Mishra, Adithya~Apuroop, and
  Sepulchre}]{mishra12a}
Mishra B, Adithya~Apuroop K, Sepulchre R (2012) A {R}iemannian geometry for
  low-rank matrix completion. Tech. rep., arXiv:1211.1550

\bibitem[{Mishra et~al.(2014)Mishra, Meyer, Bonnabel, and
  Sepulchre}]{mishra14a}
Mishra B, Meyer G, Bonnabel S, Sepulchre R (2014) Fixed-rank matrix
  factorizations and {R}iemannian low-rank optimization. Computational
  Statistics 29(3--4):591--621

\bibitem[{Moakher(2002)}]{moakher02a}
Moakher M (2002) Means and averaging in the group of rotations. SIAM Journal on
  Matrix Analysis and Applications 24(1):1--16

\bibitem[{Ngo and Saad(2012)}]{ngo12a}
Ngo TT, Saad Y (2012) {Scaled gradients on Grassmann manifolds for matrix
  completion}. In: Advances in Neural Information Processing Systems 25 (NIPS),
  pp 1421--1429

\bibitem[{Nocedal and Wright(2006)}]{nocedal06a}
Nocedal J, Wright SJ (2006) Numerical Optimization, 2nd edn. Springer Series in
  Operations Research and Financial Engineering, Springer, New York, USA

\bibitem[{Penzl(1999)}]{penzl99a}
Penzl T (1999) A cyclic low-rank {S}mith method for large sparse {L}yapunov
  equations. SIAM Journal on Scientific Computing 21(4):1401--1418

\bibitem[{Rennie and Srebro(2005)}]{rennie05a}
Rennie J, Srebro N (2005) Fast maximum margin matrix factorization for
  collaborative prediction. In: International Conference on Machine learning
  (ICML), pp 713--719

\bibitem[{Smith(1994)}]{smith94a}
Smith ST (1994) Optimization techniques on {R}iemannian manifold. In: Bloch A
  (ed) Hamiltonian and Gradient Flows, Algorithms and Control, vol~3, American
  Mathematical Society, Providence, RI, pp 113--136

\bibitem[{Theis et~al.(2009)Theis, Cason, and Absil}]{theis09a}
Theis FJ, Cason TP, Absil PA (2009) Soft dimension reduction for {ICA} by joint
  diagonalization on the {Stiefel} manifold. In: Independent Component Analysis
  and Signal Separation, Springer Berlin Heidelberg, Berlin, Germany, pp
  354--361

\bibitem[{Vandereycken(2013)}]{vandereycken13a}
Vandereycken B (2013) Low-rank matrix completion by {R}iemannian optimization.
  SIAM Journal on Optimization 23(2):1214--1236

\bibitem[{Vandereycken and Vandewalle(2010)}]{vandereycken10a}
Vandereycken B, Vandewalle S (2010) A {R}iemannian optimization approach for
  computing low-rank solutions of {L}yapunov equations. SIAM Journal on Matrix
  Analysis and Applications 31(5):2553--2579

\bibitem[{Wright and Tenny(2004)}]{wright04a}
Wright SJ, Tenny MJ (2004) A feasible trust-region sequential quadratic
  programming algorithm. SIAM Journal on Optimization 14(4):1074--1105

\end{thebibliography}

\end{document}